\documentclass[pdflatex,sn-mathphys-num]{sn-jnl}

\usepackage{graphicx}%
\usepackage{multirow}%
\usepackage{amsmath,amssymb,amsfonts}%
\usepackage{amsthm}%
\usepackage[title]{appendix}%
\usepackage{xcolor}%
\usepackage{textcomp}%
\usepackage{manyfoot}%
\usepackage{booktabs}%
\usepackage{algorithm}%
\usepackage{algorithmicx}%
\usepackage{algpseudocode}%
\usepackage{listings}%
\usepackage{mathtools}
\usepackage{enumitem}

\newcommand{\R}{\mathbb{R}}
\newcommand{\E}{\mathbb{E}}
\newcommand{\N}{\mathbb{N}}
\newcommand{\sphere}{\mathbb{S}^{d-1}}
\newcommand{\grad}{\nabla}
\newcommand{\inner}[2]{\langle #1, #2\rangle}
\newcommand{\norm}[1]{\lVert #1 \rVert}
\newcommand{\unifS}{\mathrm{Unif}(\sphere)}
\newcommand{\Id}{I_d}
\DeclareMathOperator{\tr}{tr}
\DeclareMathOperator{\Var}{Var}
\DeclareMathOperator{\rank}{rank}
\newcommand{\fmu}{f_{\mu}}
\newcommand{\gest}{\widehat{g}}
\newcommand{\design}{\mathcal{U}}
\newcommand{\bias}{\nu}
\newcommand{\bigO}{\mathcal{O}}
\newcommand{\bigOm}{\Omega}
\newcommand{\abs}[1]{\lvert #1 \rvert}

\theoremstyle{thmstyleone}%
\newtheorem{theorem}{Theorem}%
\newtheorem{proposition}[theorem]{Proposition}%
\newtheorem{lemma}[theorem]{Lemma}%

\theoremstyle{thmstyletwo}%
\newtheorem{example}{Example}%
\newtheorem{remark}{Remark}%

\theoremstyle{thmstylethree}%
\newtheorem{definition}{Definition}%
\newtheorem{assumption}{Assumption}
\newtheorem{innercustomassumption}{Assumption}
\newenvironment{customassumption}[1]
  {\renewcommand\theinnercustomassumption{#1}\innercustomassumption}
  {\endinnercustomassumption}

\begin{document}

\title[Spherical-Design Oracles for Zeroth-Order Optimisation]{Spherical-Design Oracles for Zeroth-Order Optimisation: Deterministic, Variance-Free Gradient Estimation}

\author*[1]{\fnm{Amir Ali} \sur{Farzin}}\email{amirali.farzin@anu.edu.au}

\author[2]{\fnm{Yuen-Man} \sur{Pun}}\email{yuenman.pun@unimelb.edu.au}

\author[1]{\fnm{Philipp} \sur{Braun}}\email{philipp.braun@anu.edu.au}

\author[2]{\fnm{Iman} \sur{Shames}}\email{iman.shames@unimelb.edu.au}

 \affil[1]{\orgdiv{School of Engineering}, \orgname{Australian National University}, \orgaddress{\city{Canberra}, \state{ACT}, \country{Australia}}}

 \affil[2]{\orgdiv{Department of Electrical and Electronic Engineering}, \orgname{University of Melbourne}, \orgaddress{\city{Melbourne}, \state{VIC}, \country{Australia}}}

\abstract{
In this work, we develop a zeroth-order (ZO) method that estimates the gradient with minimal bias and variance. Instead of the random directions often used in the ZO literature, we consider central differences with directions given by a fixed spherical $t$-design,
a finite point set that integrates every polynomial of degree at most $t$ exactly. The oracle is deterministic and thus does not induce any variance.
We also propose a general finite difference method in the higher-order-smoothness regime and show the bias of both finite difference methods. We then show the convergence of the proposed ZO method via the input-to-state stability concept. Numerical experiments show the superiority of our proposed method over the existing random ZO methods.
}

\maketitle

\section{Introduction}\label{sec:intro}
Consider the optimisation problem
\begin{equation}\label{eq:intro-problem}
  \min_{x\in\R^d} f(x),
\end{equation}
where $f:\R^d\to\R$ is $p\in\mathbb{N}$ times continuously differentiable for some $p\ge 2.$ In this work, we study the setting in which the objective $f$ is accessible only through function evaluations, with no analytic gradient available. For example, simulation-based engineering design \citep{conn2009introduction}. This is known as zeroth-order (ZO), or derivative-free, optimisation in the literature. 

The literature has proposed different approaches to solve the problem~\citep{larson2019derivative,audet2017derivative}. One approach is to use direct-search methods, which evaluate $f$ along a set of directions that positively spans $\R^d$ and accept a step whenever a sufficient decrease is observed~\citep{torczon1997convergence,kolda2003optimization,audet2006mesh}.
Another approach is to use model-based methods that obtain a local interpolation or regression model of $f$ on a sample set and minimise that model within a trust region~\citep{powell2006newuoa,conn2009global}. 
Both families remain active, and are supported by convergence and worst-case complexity guarantees. 

In this work, we focus on the gradient estimation methods, 
which approximates the gradients by finite differences. Such a method requires only a few function evaluations for each step and thus has found applications in machine learning and control such as black-box adversarial attacks on trained classifiers~\citep{chen2017zoo}, policy search in reinforcement learning with
non-differentiable simulators~\citep{salimans2017evolution}. For example, for the
memory-efficient fine-tuning of large language models, 
the memory overhead of backpropagation can be prohibitive and forward-pass-only ZO methods have been shown to be numerically
competitive~\citep{malladi2023fine,zhang2024revisiting}. Apart from the fewer function evaluations per step, the gradient estimation methods have two advantages. First, by approximating the gradient estimates, we can apply it to any first-order algorithms without any further modification. Second, it facilitates parallel computation, making it well-suited for high-dimensional applications.

Among this class of methods, two randomised sampled directions are often used in constructing the finite differences, namely, Gaussian distribution and the uniform distribution over a sphere~\citep{nesterov2017random,ghadimi2013stochastic,flaxman2004online,duchi2015optimal,shamir2017optimal,berahas2022theoretical}. The expectation of both gradient estimates can be expressed as the gradient of a smoothed function that inherits the convexity and Lipschitzness of $f$.
Moreover, under some standard assumptions, both methods are shown to converge to a neighbourhood of an optimum, whose radius depends on the smoothing bias and the variance of the gradient estimate.
Research is ongoing to reduce the cost under additional structural assumptions on $f$, see, e.g.,~\cite{kozak2021zeroth,choromanski2018structured,salimans2017evolution}.

Our work aims to develop a gradient estimation method that effectively approximates the true gradient with minimal bias and variance. We use the directions of a spherical design to construct finite differences. Spherical $t$-designs are finite point sets over which equal-weight averaging
integrates all polynomials of degree at most $t\in\N$ exactly on the sphere and have been often used in approximation theory and combinatorics~\cite{bannai2009survey,xiao2025spherical}. It is deterministic and does not involve any variance. In the algorithm, we compute the spherical design once and use it over all iterations. Considering it as a discrete-time dynamical system, our work shows that our method converges at the same rate as other existing randomised method but with a much smaller bias.


\subsection{Related work}

\paragraph{Gradient estimation methods.}
Finite differences has been used to estimate gradients since 1950s~\cite{kiefer1952stochastic} and random directions have been studied since 1960s~\cite{rastrigin1963convergence}. Until recently,~\cite{flaxman2004online} and~\citet{nesterov2017random} established theories for ZO methods using directions from uniform distribution over a sphere and Gaussian distribution, respectively. The latter work is particularly influential to the field and inspires stochastic variants in \cite{ghadimi2013stochastic} and \cite{duchi2015optimal}. These analyses exhibit explicit dependence on the problem dimension \(d\) in their convergence guarantees. \citet{jamieson2012query} and \citet{duchi2015optimal} establish the information-theoretic lower bounds and dimension-dependent convergence rates for some ZO oracle models in worst-case settings. 
The randomised ZO method has been analysed
for a range of structured problem classes, including quasar-convex objectives
\citep{farzin2026minimisationqc}, Polyak--{\L}ojasiewicz functions
\citep{farzin2024minimisationpl}, submodular minimisation
\citep{farzin2025minimisationsub}, and nonconvex--nonconcave min-max problems
\citep{farzin2025min,farzin2026solving}. It has also been applied to problems with different structure, such as gradient sparsity and
high-dimensional structure \citep{wang2018stochastic}, weak or effective dimension
dependence arising from a fast-decaying Hessian spectrum \citep{yue2023zeroth},
and variance reduction \citep{liu2018zeroth}.

Apart from the i.i.d. Gaussian directions and the uniform directions over a sphere, there are works that study other ZO sampling methods. The authors in \cite{kozak2021zeroth} use orthonormal random directions, whose variance of the gradient estimator is shown to be smaller than i.i.d. Gaussian directions.
Other works include
evolution-strategy variants \citep{salimans2017evolution} and orthogonal
Monte Carlo schemes \citep{choromanski2018structured}.
These methods are randomised and therefore retain sampling variance.

\paragraph{Spherical design.}
The concept of spherical design was established by
\cite{delsarte1991spherical} to study the problems in algebraic combinatorics and coding theory. They established the cardinality lower bounds
$N=\Omega(d^{\lfloor t/2\rfloor})$ that quantify the cost of reproducing the
degree-$t$ moments (see \cite{bannai2009survey} for a survey paper). Later, in~\cite{seymour1984averaging,bondarenko2013optimal}, it is shown that sphereical $t$-design exists for
every $t$ and $d$. Works have also been done to show the construction of spherical designs for different dimensions~\cite{hardin1996mclarens,womersley2018efficient}.
Spherical designs have become
standard tools in
approximation theory and numerical integration \citep{hesse2015numerical}.

\paragraph{Dynamical-system and input-to-state-stability.}
Interpreting an optimisation algorithm as a discrete-time dynamical system and
analysing its stability is a well-established approach
\citep{lessard2016analysis,wilson2021lyapunov,farzin2025properties,farzin2026stability}.
Input-to-state stability (ISS), introduced in \cite{sontag1989smooth} and
extended to discrete-time systems in \cite{sontag1996new}  {and} \cite{jiang2001input}. It
characterises the robustness of a stable equilibrium to bounded perturbations and
is widely used in nonlinear control \citep{kellett2023introduction}.
The authors in \cite{farzin2026cursed} applied ISS to investigate the ZO and first-order methods (FO) convergence difference
gap. They showed that, under suitable conditions, the averaged dynamics of a Gaussian ZO method can be written as a perturbed FO system, where the perturbation contains terms associated with the smoothing bias and the variance. Consequently, the ZO method inherits the FO decay rate and converges to a neighbourhood of the FO fixed point in expectation, whose radius depends on the dimension.

\subsection{Contributions}
Our contribution is four-fold. First, we propose the ZO spherical design method using central differences for the gradient estimation. The method is deterministic and does not involve any variance. Second, we show the bias of the proposed method with central finite differences. Third, we propose a general finite difference method to exploit the smoothness of $f\in C^p,$ which includes the central finite differences as a special case. Fourth, we develop the convergence of the proposed ZO method via the perspective of ISS. The paper is organised as follows. We start with introducing the central differences over uniform measure over a sphere and existing convergence results in Section~\ref{sec:prem}. Next, we introduce the concept of spherical design in Section~\ref{sec:prelim-design} and derive its bias in Section~\ref{subsec:central-difference}. We also propose the general finite difference method and derive its bias in Section~\ref{subsec:higher-order}. In Section~\ref{sec:conv}, we show the provable convergence of our algorithm via ISS results. We compare our bounds with the randomised Gaussian and unit sphere oracle in Section~\ref{sec:randomised-oracle-error}. Lastly, we demonstrate the superiority of our algorithm in Section~\ref{sec:numerics} and conclude our paper with future directions in Section~\ref{sec:discussion}.

\subsection*{Notation}
Throughout, $d\in \mathbb{N}$ denotes the dimension  {of $x$ in \eqref{eq:intro-problem}} and $\R^d$ the Euclidean space with
inner product $\inner{a}{b}=a^\top b$ and norm $\norm{a}=\inner{a}{a}^{1/2}$. The
same symbol $\norm{\cdot}$ denotes the spectral norm of a matrix. For $z\in\R^d$
and $r>0$, $B_r(z):=\{x\in\R^d:\norm{x-z}\le r\}$ is the closed ball of radius $r$
about $z$  {and}  {$B_r$ denotes the closed ball of radius $r$ about the origin.} For a square-integrable random vector $Z$, we use
\(
  \|Z\|_{L^2}
  = \bigl(\E[\|Z\|^2]\bigr)^{1/2}.
\) We write $\Id$ for the $d\times d$ identity  {matrix}, $e_1,\dots,e_d$ for the
standard basis, and $\delta_{ab}$ for the Kronecker delta ($\delta_{ab}=1$ if
$a=b$, else $0$).  For a tensor $T$ of order $k$ and indices $a_1,\dots,a_k\in\{1,\dots,d\}$, we
write $(T)_{a_1\cdots a_k}$ for the corresponding entry of $T.$ {The unit sphere is $\sphere=\{u\in\R^d:\norm{u}=1\}$, equipped
with the uniform (rotation-invariant) probability measure $\sigma$}.  {The symbol} $\unifS$
denotes  {the distribution of vectors uniformly sampled from $\sphere$}, and $\mathcal N(0,\Id)$ the standard
Gaussian. Expectation is  {denoted by} $\E[\cdot]$ and variance $\Var(\cdot)$. For a matrix,
$\tr(\cdot)$ and $\rank(\cdot)$ are its trace and rank.
For a sufficiently smooth $f:\R^d\to\R$, $\grad f(x)\in\R^d$ is the gradient,
$\grad^2 f(x)\in\R^{d\times d}$ the Hessian, and $\Delta f(x)=\tr\grad^2 f(x)=\sum_{a=1}^{d}\partial^2 f(x)/\partial x_a^2$ the Laplacian.
More generally $\grad^k f(x)$ is the
order-$k$ derivative tensor with entries
$(\grad^k f(x))_{a_1\cdots a_k}=\partial^k f(x)/(\partial x_{a_1}\cdots\partial
x_{a_k})$.  {We denote the tensor (outer) product by $\otimes.$} For a vector $u$ we write $u^{\otimes k}=u\otimes\cdots\otimes u$
($k$ factors) and denote by $\grad^k f(x)[u^{\otimes k}]$ the full contraction
\[
  \grad^k f(x)[u^{\otimes k}]
  =\grad^k f(x)[\underbrace{u,\dots,u}_{k}]
  =\sum_{a_1,\dots,a_k=1}^{d}
  (\grad^k f(x))_{a_1\cdots a_k}\,u_{a_1}\cdots u_{a_k}\ \in\R.
\]
We use the operator norm $\norm{\grad^k f(x)}=\sup\{\,\abs{\grad^k f(x)[v_1,\dots,v_k]}:
\norm{v_1}=\dots=\norm{v_k}=1\,\}$. The class $C^p$ consists of $p$-times
continuously differentiable functions.  {We denote the class of continuously differentiable functions having Lipschitz gradients with $C^{1,1}.$}
We write $\bigO(\cdot)$ and $\bigOm(\cdot)$ for the usual asymptotic upper- and
lower-bound notation, $a=\Theta(b)$ when $a=\bigO(b)$ and $b=\bigO(a).$ For nonnegative
integers $n\ge k$, $\binom{n}{k}=\tfrac{n!}{k!\,(n-k)!}$ is the binomial
coefficient. For a real $x$, $\lfloor x\rfloor$ and $\lceil x\rceil$ are the floor
and ceiling. For a symmetric matrix $A$, $A\succ 0$ means $A$ is positive
definite.
A continuous function
$\alpha:[0,\infty)\to[0,\infty)$ is of class $\mathcal K$ if it is strictly
increasing with $\alpha(0)=0$, and of class $\mathcal K_\infty$ if in addition
$\alpha(s)\to\infty$ as $s\to\infty$. A continuous function
$\beta:[0,\infty)\times[0,\infty)\to[0,\infty)$ is of class $\mathcal{KL}$ if
$s\mapsto\beta(s,k)$ is of class $\mathcal K$ for each fixed $k$ and
$k\mapsto\beta(s,k)$ is decreasing to $0$ as $k\to\infty$ for each fixed $s$. Lastly, we denote the all-ones vector as $\mathbf 1,$ whose dimension will be clear from context.

\section{Preliminaries}\label{sec:prem}

\subsection{Unit ball smoothing}\label{sec:prelim-smooth}
Throughout the paper, we make the following assumption on $f.$
\begin{assumption}\label{ass:smooth}
 We assume that the objective function $f \in C^{p}(\R^d)$, for some integer $p \ge 2$, is bounded below,
$f^\star := \inf_x f(x) > -\infty$, and has $L$-Lipschitz gradient.
\end{assumption}
To minimise $f$ with access only to the function evaluations, one may approximate the gradient at $x\in\R^d$ via 
\begin{equation*}
    g(x,u)=
    \begin{cases}
        \frac{f(x+\mu u)-f(x)}{\mu}\, u ,\quad &\text{(forward differences)}\\
        \frac{f(x+\mu u)-f(x-\mu u)}{2\mu}\, u, \quad &\text{(central differences)}
    \end{cases}
\end{equation*}
along some direction $u\in\R^d$ for some $\mu>0$. Suppose that we sample $u$ from a uniform distribution over a unit sphere. For both finite difference methods, its expectation equals the gradient of the uniform smoothed function. Specifically, let $f_\mu\colon\R^d\to\R$ be the smoothed function 
$\fmu(x):=\E_{u\sim\mathrm{Unif}(\mathbb{B}^d)} f(x+\mu u),$ where $\mathbb{B}^d$ is the unit ball on $\R^d$ and $\mu>0$ is the smoothing radius.
We have~\citep{flaxman2004online}
\begin{equation}\label{eq:smoothed-grad}
  \grad \fmu(x)
  \;=\; d\int_{\sphere}
  g(x,u) \mathrm{d}\sigma(u).
\end{equation}
In this work, we focus on the \emph{central differences}. For $f\in C^3,$ the gradient estimation error then involves the third derivative of $f;$ i.e., there exists $\xi\in(-\mu,\mu)$
such that
\begin{equation}\label{eq:cdiff}
  \frac{f(x+\mu u)-f(x-\mu u)}{2\mu}
   = \inner{\grad f(x)}{u} + \tfrac{\mu^2}{6}\,\grad^3 f(x+\xi u)[u,u,u].
\end{equation}

We compute the moments of the uniform measure, which will be used in our analysis. Specifically, let \(\mathbb{M}_i := \int_{\mathbb{S}^{d-1}} u^{\otimes i}\,\mathrm{d}\sigma(u)\) be the \(i\)-th moment of the uniform measure \(\sigma\) on \(\mathbb{S}^{d-1}\). Since $\sigma$ is invariant
under $u \mapsto -u$, every odd moment vanishes,
$\mathbb{M}_i = 0,\; \forall i=2k-1, k\in\mathbb{N}.$ The
second and fourth moments are given by~\cite{folland2001integrate}
\begin{align}
  &\mathbb{M}_2=\int_{\mathbb{S}^{d-1}} u u^{\top}\,\mathrm{d}\sigma(u)
    = \tfrac{1}{d} I_d \;\in\; \mathbb{R}^{d \times d},
    \label{eq:moment2}\\
  &(\mathbb{M}_4)_{abcd}=\int_{\mathbb{S}^{d-1}} \!\!\!\!\!u_a u_b u_c u_e \mathrm{d}\sigma(u)
    = \frac{\delta_{ab}\delta_{ce} + \delta_{ac}\delta_{be}
             + \delta_{ae}\delta_{bc}}{d(d+2)} \in \mathbb{R},
    \;\; a,b,c,e \in \{1,\dots,d\}.
    \label{eq:moment4}
\end{align}

\subsection{Dynamical systems and input-to-state stability}\label{sec:prelim-iss}

A ZO method could be viewed as its first-order (FO) counterpart along with
an additive gradient estimation error. Employing the notion of input-to-state stability~\cite{farzin2026cursed}, we can address how far the perturbed ZO trajectory deviates from the unperturbed one.

Specifically, consider an FO algorithm $w:\R^d\to\R^d$ that updates an iterate $z_k\in\R^d$ by $z_{k+1}=w(z_k)$ at the $k$-th iteration. Suppose that there exists an equilibrium $z^e\in\R^d$ satisfying $w(z^e)=z^e$. Its perturbed ZO
counterpart can be written as
\begin{equation}\label{eq:perturbed-system}
  z_{k+1}=w(z_k)+q_k ,
\end{equation}
where $\{q_k\}$ is the ZO perturbation sequence. We say that the system \eqref{eq:perturbed-system} is locally input-to-state stable (ISS) about $z^e$ if there exist $\delta>0$,
$\beta\in\mathcal{KL}$, and $\gamma\in\mathcal K$ such that, for every perturbation
sequence and every $z_0\in B_\delta(z^e)$,
\begin{equation}\label{eq:iss-estimate}
  \norm{z_k-z^e}\;\le\;\beta\bigl(\norm{z_0-z^e},k\bigr)
  +\gamma\Bigl(\sup_{0\le j\le k}\norm{q_j}\Bigr),
  \qquad k\in\mathbb N .
\end{equation}
The estimate \eqref{eq:iss-estimate} reduces to asymptotic stability when
$q_k\equiv0$ and, otherwise, guarantees convergence to a neighbourhood of $z^e$
whose radius is governed by the perturbation magnitude
$\sup_j\norm{q_j}$. The following result states that an algorithm whose
unperturbed dynamics admit a Lyapunov certificate inherits the ISS estimate  {if the perturbations are uniformly bounded}. It is
the mechanism by which the convergence guarantees of a FO method are
 captured in its perturbed, ZO counterpart.

\begin{proposition}[{\cite[Thms.~1--2]{farzin2026cursed}}]
\label{prop:iss-transfer}
Let $w\colon\R^d\to\R^d$, and let $z^e\in\R^d$ be an equilibrium of $w$, i.e.,
$w(z^e)=z^e$. Suppose there exist a radius $R>0$, a locally Lipschitz function
$V\colon\R^d\to\R_{\ge0}$, and functions $\alpha_1,\alpha_2,\alpha_3\in\mathcal K_\infty$
such that, for all $z\in B_R(z^e)$,
\begin{equation}\label{eq:lyap-conditions}
  \alpha_1(\norm{z-z^e})\le V(z)\le \alpha_2(\norm{z-z^e})
  \quad\text{and}\quad
  V(w(z))-V(z)\le -\alpha_3(\norm{z-z^e}),
\end{equation}
which means that the equilibrium $z^e$ of $z_{k+1}=w(z_k)$ is
  locally asymptotically stable: there exist $\delta>0$ and $\beta\in\mathcal{KL}$
  such that $\norm{z_k-z^e}\le\beta(\norm{z_0-z^e},k)$ for all $k\ge0$ and all
  $z_0\in B_\delta(z^e)$. Then the following hold.
\begin{enumerate}[label=(\roman*)]
  \item \emph{Robustness to perturbations:} There exist radii $\bar r,\bar q>0$ and
  functions $\beta\in\mathcal{KL}$, $\gamma\in\mathcal K$ such that, for every
  $z_0\in B_{\bar r}(z^e)$ and every perturbation sequence with
  $\sup_k\norm{q_k}\le\bar q$, the solution of \eqref{eq:perturbed-system}
  satisfies \eqref{eq:iss-estimate} with the $\beta$ function as the unperturbed system.
  \item \emph{Exponential case:} If, in addition, there exist $c_1,c_2>0$ and
  $c\in(0,1)$ such that, for all $z\in B_R(z^e)$,
  \[
    c_1\norm{z-z^e}^2\le V(z)\le c_2\norm{z-z^e}^2
    \quad\text{and}\quad
    V(w(z))-V(z)\le -c\,V(z),
  \]
  then (ii) holds with $\beta(s,k)=\sqrt{c_2/c_1}\,s\,(1-c)^{k/2}$. In particular,
  the perturbed system decays at the same exponential rate as the unperturbed one.
\end{enumerate}
\end{proposition}


\section{Spherical Design Oracle}\label{sec:oracle}
Their convergence bounds are in terms of the bias and the variance of the ZO gradient estimator, which do not diminish over iterations. We propose a ZO gradient oracle that uses a spherical $t$-design to construct the directions of central finite differences.

\subsection{Spherical design}\label{sec:prelim-design}
Let us start with the definition of spherical design~\cite{delsarte1991spherical}.
\begin{definition}\label{def:design}
A finite set $\design=\{u_1,\dots,u_N\}\subset\sphere$ is a \emph{spherical
$t$-design} if
\begin{equation}\label{eq:design}
  \frac1N\sum_{i=1}^N P(u_i) = \int_{\sphere} P (u) \, \mathrm{d}\sigma(u),
\end{equation}
for every polynomial $P:\R^d\rightarrow\R$ of degree at most $t.$
\end{definition}

Spherical designs are finite point sets on the unit sphere whose equal-weight averages reproduce the uniform measure exactly on all polynomials up to a given degree. It is widely used in approximation theory and combinatorics; see, e.g.,~\cite{bannai2009survey,xiao2025spherical}.

It is known that spherical $t$-designs exist for every $t$ and $d$
\citep{seymour1984averaging,bondarenko2013optimal} and they are not unique:
Since $\sigma$ is rotation invariant, $Q\design$ is a $t$-design for every orthogonal $Q$ whenever $\design$ is. Below are examples of 2- and 3-designs.

\begin{example}[Regular simplex, $t=2$, $N=d+1$]\label{ex:simplex}
Let
$Q\in\mathbb R^{(d+1)\times d}$ have orthonormal columns spanning
$\mathbf 1^\perp$ (for instance, obtained from a QR factorisation). Then
\[
u_i=\sqrt{\frac{d+1}{d}}\,
Q^\top\left(e_i-\frac{1}{d+1}\mathbf 1\right),
\qquad i=1,\dots,d+1,
\]
are unit vectors satisfying $\sum_{i=1}^{d+1}u_i=0$ and $\frac{1}{d+1}\sum_{i=1}^{d+1}u_i u_i^\top=\frac1d I_d.$
Hence, $\design$ is a spherical $2$-design. Such a design can be computed in $\bigO(d^2)$
operations. In fact, a set of unit vectors is a
$2$-design if and only if $\sum_i u_i=0$ and
$\frac1N\sum_i u_iu_i^\top=\frac1d I_d$.
\end{example}

\begin{example}[Cross-polytope, $t=3$, $N=2d$]\label{ex:cross}
The set $\{\pm e_i\}_{i=1}^d$ is a spherical $3$-design. Its odd moments
vanish by antipodal symmetry, while
\[
\frac{1}{2d}\sum_{i=1}^d
\left(e_i e_i^\top+(-e_i)(-e_i)^\top\right)
=\frac1d I_d.
\]
\end{example}

Explicit constructions of $4$- and $5$-designs for every $d$ are given
in \cite{bajnok1991construction}. Readers can also find more numerically computed designs for
moderate $d$ and $t$ in
\cite{hardin1996mclarens,womersley2018efficient}. \

The average over the spherical design $\design$ in~\eqref{eq:design} reproduces the moments
of the uniform distribution over a sphere exactly for every order $i\le t$. However, this
exactness has a price: The number of directions, and hence of function
queries, grow with both the degree $t$ and the dimension $d,$ as shown in the next proposition.

\begin{proposition}[\citep{delsarte1991spherical,bannai2009survey}]\label{prop:cost}
Let $l=\lfloor t/2\rfloor$. Every spherical $t$-design  {on
$S^{d-1}\subset\R^d$} satisfies
$N\ge\binom{d-1+l}{l}+\binom{d-2+l}{l-1}$ for $t=2l$ and
$N\ge 2\binom{d-1+l}{l}$ for $t=2l+1$.
\end{proposition}

From Proposition~\ref{prop:cost}, for fixed $t,$
a spherical $t$-design requires at least
$N \ge \bigOm\bigl(d^{\,\lfloor t/2\rfloor}\bigr)$ directions. Both Examples~\ref{ex:simplex} and~\ref{ex:cross} attain the
lower bound in Proposition~\ref{prop:cost}. Designs of small size also arise as
orbits of finite reflection groups, for example the normalised minimal
vectors of the $E_8$ lattice ($d=8$, $N=240$, $t=7$) from \citep{bannai2009survey}, which will be used in
Section~\ref{sec:numerics}.

\subsection{Central difference oracle}\label{subsec:central-difference}

Given a spherical $t$-design $\design=\{u_i\}_{i=1}^N,$ our central difference design oracle is given by
\begin{equation}\label{eq:base-oracle}
  \gest(x) \;=\; \frac{d}{2N\mu}\sum_{i=1}^{N}
  \bigl(f(x+\mu u_i)-f(x-\mu u_i)\bigr)\, u_i,
\end{equation}
that is, the central finite-difference gradient estimator along the directions of the spherical design $\design$. Each call costs $2N$ function evaluations. With the cross-polytope of Example~\ref{ex:cross}, the oracle
\eqref{eq:base-oracle} reduces to coordinate-wise central finite
differences
using $2d$ distinct
evaluations. Given a fixed spherical design $\design,$ the oracle~\eqref{eq:base-oracle} is deterministic and hence without estimator variance.


To implement it, we first construct a spherical $t$-design and store it for every iteration. With an initial guess $x_0,$ at each iteration, we update the iterate $x_k$ by a ZO gradient descent step with size $\eta>0$
\begin{equation}\label{eq:gd}
  x_{k+1} = x_k - \eta\, \gest(x_k).
\end{equation}

Next, let us analyse some properties of the oracle \eqref{eq:base-oracle}.
\begin{lemma}[Properties of the oracle \eqref{eq:base-oracle}]\label{thm:exact}
Let $\design=\{u_i\}_{i=1}^N\subset\sphere$ be a spherical $t$-design with
$t\ge 2$ and $\gest$ be the oracle \eqref{eq:base-oracle}. Then the following hold:
\begin{enumerate}[label=(\roman*)]
\item The design satisfies
$\displaystyle \frac1N\sum_{i=1}^N u_iu_i^\top=\frac1d\,\Id $ and 
for every differentiable $f$ and every $x\in\R^d$,
\[
  \frac{d}{N}\sum_{i=1}^N \inner{\grad f(x)}{u_i}\,u_i
  = \grad f(x).
\]
\item Consider a polynomial $q$ with $\deg q\le 2\lfloor t/2\rfloor.$ Define the 
smoothed gradient operator
\[
  \mathcal{G}_\mu(x)\;:=\;d\int_{\sphere}
  \frac{q(x+\mu u)-q(x-\mu u)}{2\mu}\,u\,\mathrm{d}\sigma(u).
\]
Then
$\gest(x)=\mathcal{G}_\mu(x)$ for every $x\in\R^d$ and $\mu>0$.
\end{enumerate}
\end{lemma}

\begin{proof}
\emph{(i)}
Since $u\mapsto uu^\top$ is a matrix-valued polynomial of degree $2\le t$,
the spherical $t$-design property \eqref{eq:design} and \eqref{eq:moment2} gives
\[
  \frac1N\sum_{i=1}^N u_i u_i^\top
  =
  \int_{\sphere} uu^\top\,\mathrm{d}\sigma(u) =
  \frac1d\,\Id.
\]

Right-multiplying above by $\grad f(x)$ and scaling by $d$ gives
$\frac{d}{N}\sum_i\inner{\grad f(x)}{u_i}u_i=d(\tfrac1N\textstyle\sum_i u_i u_i^\top)\grad f(x)=d\cdot\tfrac1d\,\Id\,\grad f(x)
=\grad f(x)$. 

\emph{(ii)} Fix $x$ and $\mu$. Since
$q(x+\mu u)=\sum_{k=0}^{\deg q}\frac{\mu^k}{k!}\grad^k q(x)[u^{\otimes k}]$,
the central difference
$D(u):=\frac{q(x+\mu u)-q(x-\mu u)}{2\mu}$ is a polynomial in $u$ containing
only the odd-degree terms, so its degree is at most the largest odd
$k\le\deg q$. Each component of $u\mapsto D(u)\,u$ is therefore a polynomial
of degree at most $2\lceil \deg q/2\rceil\le 2\lfloor t/2\rfloor\le t$, and
\eqref{eq:design} applied componentwise, multiplied by $d$, gives
$\gest[q](x)=\mathcal{G}_\mu[q](x)$.

\end{proof}

Lemma~\ref{thm:exact} justifies the use of the spherical design for estimating the gradients. Also, it shows that, for $l\ge1,$ a $2l$-design and a $(2l+1)$-design both satisfy
\[
  \gest(x)=\mathcal{G}_\mu(x)
\]
for every polynomial $q$ with $\deg q\le 2l.$ In fact, raising the degree from $2l$ to $2l+1$ of the spherical design does not reduce the bias of the gradient estimate. This is because the Taylor expansion of $q(x+\mu u)-q(x-\mu u)$ only contains odd powers
of $u$, so the terms $(q(x+\mu u)-q(x-\mu u))u$ in the oracle \eqref{eq:base-oracle} is an even polynomial in $u$, which odd-order moments never appear.

Next, we bound the smoothing bias of the oracle \eqref{eq:base-oracle}. To proceed, let $x_0\in\R^d$ denote a point and let
\begin{equation}
  M_j(x_0,\Delta_0,\mu) \;:=\; \sup_{x \in \mathcal M(x_0,\Delta_0,\mu)} \norm{\grad^{j} f(x)},
  \qquad 2 \le j \le p ,
  \label{eq:Mj}
\end{equation}
be the higher-order derivatives bounds over $\mathcal M (x_0,\Delta_0,\mu),$ where
\begin{equation}
  \mathcal M (x_0,\Delta_0,\mu)\;:=\; \bigl\{\, x \in \R^d :
  \operatorname{dist}\bigl(x, \{\, y : f(y) \le f(x_0) + \Delta_0 \,\}\bigr)
  \le \mu \,\bigr\}
  \label{eq:region-N}
\end{equation}
for some margin $\Delta_0 > 0$ and some radius $\mu>0.$ We see that $M_j(x_0,\Delta_0,\mu)$ is non-decreasing in $\Delta_0$ or $\mu$.

\begin{customassumption}{2$'$}\label{ass:finite-M-base}
$M_j(x_0,\Delta_0,\mu)$ defined in \eqref{eq:Mj} is finite for $j\in\{2,\dots,p\}$.
\end{customassumption}
\setcounter{assumption}{1} 

 {
If $f$ is a coercive function (i.e., $\lim_{\|x\|\rightarrow\infty} f(x) = \infty$), then Assumption~\ref{ass:finite-M} is satisfied for any $x_0\in\R^d$ and finite $\Delta_0,\,\mu$ with a compact set $\mathcal{M}(x_0,\Delta_0,\mu).$ 

Now, we are ready to bound the smoothing bias of the oracle \eqref{eq:base-oracle}.

\begin{lemma}[Bias of the oracle \eqref{eq:base-oracle}]\label{lem:bias}
Suppose that Assumptions~\ref{ass:smooth} and~\ref{ass:finite-M-base} hold. Let $\design$ be a spherical $t$-design and $\gest$ be the oracle \eqref{eq:base-oracle} with smoothing radius $\mu>0$ in \eqref{eq:base-oracle}. For any $x_0\in\R^d,$ $\Delta_0>0,$ let
$x\in\mathcal M(x_0,\Delta_0,\mu)$ and write $\gest(x)=\grad f(x)+\bias(x).$
\begin{enumerate}[label=(\roman*)]
\item If $t\ge2$ and $p\ge3$, then
$\displaystyle \norm{\bias(x)}\le\frac{d}{6}\,\mu^{2}\,M_3(x_0,\Delta_0,\mu).$
\item If $t\ge4$ and $p\ge5$, then,
$\displaystyle \norm{\bias(x)}\le\frac{d\,\mu^2 M_3(x_0,\Delta_0,\mu)}{2(d+2)}+\frac{d}{120}\,\mu^4 M_5(x_0,\Delta_0,\mu).
$
\end{enumerate}
\end{lemma}

\begin{proof}



\emph{(i)} Fix $x\in\mathcal M(x_0,\Delta_0,\mu),$ let $M_3 = M_3(x_0,\Delta_0,\mu),$ and $u\in\sphere$, and set $g(s):=f(x+su)$. Since
$g\in C^3$, using Taylor's theorem and intermediate value theorem gives, for some
$\xi=\xi(u)\in(-\mu,\mu)$,
\[
  \frac{f(x+\mu u)-f(x-\mu u)}{2\mu}
  =\inner{\grad f(x)}{u}+\frac{\mu^2}{6}\,\grad^3 f(x+\xi u)[u^{\otimes3}].
\]
Using the definition of $\widehat{g}$ in~\eqref{eq:base-oracle} and Lemma~\ref{thm:exact}(i) gives
\[
  \bias(x)=\frac{d\,\mu^2}{6N}\sum_{i=1}^N
  \grad^3 f(x+\xi_i u_i)[u_i^{\otimes3}]\,u_i,
  \qquad \xi_i:=\xi(u_i).
\]
Each probe point $x+\xi_i u_i$ lies in $\mathcal M(x_0,\Delta_0)$ and $\norm{u_i}=1$, so the
operator-norm bound gives
$\norm{\grad^3 f(\cdot)[u_i^{\otimes3}]\,u_i}\le M_3$. Averaging over the $N$
directions yields $\norm{\bias(x)}\le \frac{d}{6}\mu^2 M_3$.

\emph{(ii)} Let $M_5=M_5(x_0,\Delta_0,\mu).$ When $g\in C^5,$ expanding  $g$ to fifth order via
Taylor's theorem 
and considering the intermediate value theorem gives, for some $\theta=\theta(u)\in(-\mu,\mu)$,
\[
  \frac{f(x+\mu u)-f(x-\mu u)}{2\mu}
  =\inner{\grad f(x)}{u}+\frac{\mu^2}{6}\,\grad^3 f(x)[u^{\otimes3}]
  +\frac{\mu^4}{120}\,\grad^5 f(x+\theta u)[u^{\otimes5}].
\]
Apply $\frac{d}{N}\sum_i(\cdot)\,u_i$. The linear term reproduces $\grad f(x)$ by
Lemma~\ref{thm:exact}(i). The cubic term has integrand
$u\mapsto\grad^3 f(x)[u^{\otimes3}]u_j$, a polynomial of degree $4\le t$ evaluated
at the fixed tensor $\grad^3 f(x)$, so by Lemma~\ref{thm:exact}(ii) the
$4$-design integrates it exactly:
\[
  \frac{d}{N}\sum_{i=1}^N\grad^3 f(x)[u_i^{\otimes3}]\,u_i
  =d\!\int_{\sphere}\!\grad^3 f(x)[u^{\otimes3}]\,u\,\mathrm{d}\sigma(u) .
\]
We evaluate this integral componentwise. Its $j$-th component is
\[
  d\int_{\sphere}\grad^3f(x)[u^{\otimes3}]\,u_j\,\mathrm d\sigma(u)
  =d\sum_{a,b,c}\partial_{abc}f(x)
  \int_{\sphere}u_au_bu_cu_j\,\mathrm d\sigma(u),
\]
and substituting the fourth-moment identity \eqref{eq:moment4},
\[
  \int_{\sphere}u_au_bu_cu_j\,\mathrm d\sigma(u)
  =\frac{\delta_{ab}\delta_{cj}+\delta_{ac}\delta_{bj}+\delta_{aj}\delta_{bc}}
        {d(d+2)},
\]
Each Kronecker pair collapses two of the three summation indices, leaving
\[
  \frac{1}{d+2}\left(
  \sum_{a=1}^d\partial_{aaj}f(x)
  +\sum_{a=1}^d\partial_{aja}f(x)
  +\sum_{b=1}^d\partial_{jbb}f(x)\right).
\]
Because $\grad^3 f$ is symmetric, the three sums coincide, each equal to
\[
  \sum_{a=1}^d\partial_{aaj}f(x)
  =\partial_j\!\Bigl(\sum_{a=1}^d\partial_{aa}f(x)\Bigr)
  =\partial_j(\Delta f)(x).
\]
Collecting components,
\[
  \frac{d}{N}\sum_{i=1}^N\grad^3f(x)[u_i^{\otimes3}]\,u_i
  =d\!\int_{\sphere}\!\grad^3f(x)[u^{\otimes3}]\,u\,\mathrm d\sigma(u)
  =\frac{3}{d+2}\,\grad(\Delta f)(x),
\]
and multiplying by the Taylor coefficient $\mu^2/6$ yields the leading bias term
$\tfrac{\mu^2}{2(d+2)}\grad(\Delta f)(x)$.
 For any unit vector $a$,
\[
  \bigl|\inner{\grad(\Delta f)(x)}{a}\bigr|
  =\Bigl|\sum_{b=1}^d\grad^3f(x)[e_b,e_b,a]\Bigr|
  \le\sum_{b=1}^d\bigl|\grad^3f(x)[e_b,e_b,a]\bigr|
  \le\sum_{b=1}^d\norm{\grad^3f(x)}=d\,M_3,
\]
where the last inequality uses
$|\grad^3f(x)[e_b,e_b,a]|\le\norm{\grad^3f(x)}\,\norm{e_b}^2\norm{a}
=\norm{\grad^3f(x)}$. Taking the supremum over unit $a$ gives
$\norm{\grad(\Delta f)(x)}\le d\,M_3$.

The quintic remainder $r(x)$ is bounded as in part~(i): each probe point
$x+\theta u_i$ lies in $\mathcal M(x_0,\Delta_0)$ and $\norm{u_i}=1$, so
$\norm{\grad^5 f(x+\theta u_i)[u_i^{\otimes5}]\,u_i}\le M_5$, whence
\[
  \norm{r(x)}\le\frac{d\mu^4}{120N}\sum_{i=1}^N M_5=\frac{d}{120}\mu^4 M_5 .
\]

Consequently,
\[
  \norm{\bias(x)}\le\frac{d\mu^2 M_3}{2(d+2)}+\frac{d}{120}\mu^4 M_5
  =\bigO(\mu^2 M_3).
\]
\end{proof}

The lemma shows that
a $4$-design reduces the leading constant by a factor of $d$.
 
\begin{remark}[$C^{1,1}$ case]\label{rem:c11-bias}
Lemma~\ref{lem:bias} considers $f\in C^p$ where $p\ge 3$ (or $p\ge5$). For $f\in C^{1,1},$ where the second and the third derivatives may not exist, a weaker bound still follows from the Lipschitz gradient
alone. Specifically, the oracle \eqref{eq:base-oracle} $\gest = \nabla f + \nu$ with ZO error $\nu\colon\R^d\to\R$ satisfies
\begin{equation}\label{eq:c11-bias-bound}
\norm{\bias(x)}\;\le\;\frac{d}{2}\,\mu\,L, 
\end{equation}
which is worse than $f\in C^3$ by an order of $\mu$ (cf. Lemma~\ref{lem:bias}(i)). To see this, fix a direction $u\in\sphere$ and set $g(s):=f(x+su)$, so that
$g\in C^1$ with $g'(s)=\inner{\grad f(x+su)}{u}$. Then,
\[
  f(x+\mu u)=f(x)+\int_0^\mu\inner{\grad f(x+su)}{u}\,\mathrm ds,
  \qquad
  f(x-\mu u)=f(x)-\int_0^\mu\inner{\grad f(x-su)}{u}\,\mathrm ds,
\]
the second identity following from the substitution $s\mapsto-s$. Subtracting and
dividing by $2\mu$, the two values of $f(x)$ cancel, and writing the constant
$\inner{\grad f(x)}{u}$ as its own average over $[0,\mu]$ gives the exact
central-difference error
\begin{align*}
     \frac{f(x+\mu u)-f(x-\mu u)}{2\mu}&-\inner{\grad f(x)}{u}
  =\\
  &\frac{1}{2\mu}\int_0^\mu\!\Bigl(
     \inner{\grad f(x+su)-\grad f(x)}{u}
    +\inner{\grad f(x-su)-\grad f(x)}{u}\Bigr)\,\mathrm ds .
\end{align*}

By the Cauchy--Schwarz inequality, $\norm{u}=1$, and the $L$-Lipschitz gradient,
each inner product is bounded in absolute value by
$\norm{\grad f(x\pm su)-\grad f(x)}\le L\,s$, so the integrand is at most $2Ls$ and
\[
  \Bigl|\tfrac{1}{2\mu}\bigl(f(x+\mu u)-f(x-\mu u)\bigr)-\inner{\grad f(x)}{u}\Bigr|
  \;\le\;\frac{1}{2\mu}\int_0^\mu 2Ls\,\mathrm ds
  \;=\;\frac{\mu L}{2}.
\]
Applying the estimator map $\tfrac{d}{N}\sum_{i}(\cdot)\,u_i$, using
Lemma~\ref{thm:exact}(i) to reproduce the linear part exactly, and bounding the
$N$ residuals by $\tfrac{d}{N}\sum_i\tfrac{\mu L}{2}\norm{u_i}\leq\tfrac{d}{2}\mu L$
yields \eqref{eq:c11-bias-bound}.

\end{remark}

\begin{remark}[Central versus forward differences]\label{rem:central-vs-forward}
One may use forward difference as the oracle in~\eqref{eq:base-oracle} instead, but the bias would be different from the central-difference oracle.
Specifically, consider the Taylor expansion of the forward-difference spherical $t$-design oracle:
\begin{align}
    \frac{d}{N}\sum_{i=1}^N \frac{f(x+\mu u_i)-f(x)}{\mu}u_i
  = &\underbrace{\frac{d}{N}\sum_{i=1}^N\inner{\grad f(x)}{u_i}u_i}_{\mathrm{(I)}}+\underbrace{\frac{d}{N}\sum_{i=1}^N\frac{\mu}{2}\,\grad^2 f(x)[u_i,u_i]u_i}_{\mathrm{(II)}} \nonumber\\
  &\;\;\;\;+\underbrace{\frac{d}{N}\sum_{i=1}^N\frac{\mu^2}{6}\,\grad^2 f(x+\xi_i u_i)[u_i,u_i,u_i]u_i}_{\mathrm{(III)}} \label{eq:central-vs-forward}
\end{align}
for some $\xi_i\in(0,\mu),\, i=1,\ldots,N.$ With $t\ge 2,$ we have the first term (I) equals the true gradient $\nabla f(x)$ and leave (II) and (III) as bias terms. For a 2-design oracle, the term (II) is of the third moment of $u_i$ and cannot be cancelled out, yielding a bias of $\bigO(d\,\mu\,L).$
To achieve $\mathcal{O}(\mu^2)$ bias, we need a 3-design oracle such that (II) vanishes because of the odd symmetry of the uniform distribution (i.e., $\int_{\sphere}u_au_bu_c\,\mathrm d\sigma(u)=0$). In contrast, the central-difference oracle cancels the even-order terms (corresponding to the term (II) in~\eqref{eq:central-vs-forward}) by definition and thus only needs a $2$-design to attain $\bigO(\mu^2)$ bias.
\end{remark}

\subsection{Higher-order oracle}\label{subsec:higher-order}
To exploit smoothness beyond second order, we augment the design and give the \emph{higher-order} oracle
\begin{equation}\label{eq:ho-oracle}
  \gest^{(p)}(x) \;=\; \frac{d}{N\mu}\sum_{i=1}^{N}\sum_{r=1}^{R}
  c_r\, f(x+\mu \rho_r u_i)\, u_i
\end{equation}
for radial difference stencil of nodes $\{\rho_r\}_{r=1}^{R} \subset \R$ and weights $\{c_r\}_{r=1}^{r=R}\subset \R$ with $R\in\N$ number of nodes. In particular, if we take the single node pair $\rho=\pm1$, $c=\pm\tfrac12,$ it reduces to the oracle \eqref{eq:base-oracle}, which is also the lowest-order instance of \eqref{eq:ho-oracle}. Using the higher-order oracle, we can do the ZO update as~\eqref{eq:gd} but replace $\gest$ with $\gest^{(p)}.$

Let $\rho_{\max}\ge |\rho_r|$ for $r=1,\ldots,R.$ Similar to Assumption~\ref{ass:finite-M-base}, we need an assumption that bounds the higher-order derivatives.
\begin{assumption}\label{ass:finite-M}
    $M_j(x_0,\Delta_0,\mu\rho_{\max})$ defined in \eqref{eq:Mj} is finite for $j\in\{2,\dots,p\}$.
\end{assumption}
As a special case, $\rho_{\max}=1$ for the central differences. Similarly to the oracle \eqref{eq:base-oracle}, the higher-order oracle~\eqref{eq:ho-oracle} is deterministic with a fixed spherical design $\design.$ The next theorem shows how we can construct a general $(\rho_r,c_r).$

\begin{theorem}\label{thm:ho}
Suppose that Assumptions~\ref{ass:smooth} and~\ref{ass:finite-M} hold. Let $\design$ be a spherical $t$-design with $t\ge 2$ and $R=2\lfloor p/2\rfloor.$ Let
$\{\rho_r\}_{r=1}^{R}\subset\R\setminus\{0\}$ be a symmetric node set
 {($\rho\in\{\rho_r|r\in\{1,\dots,R\}\}\Rightarrow-\rho\in\{\rho_r|r\in\{1,\dots,R\}\}$)} 
with antisymmetric weights $\{c_r\}$
($c_{r'}=-c_r$ whenever $\rho_{r'}=-\rho_r$) satisfying the order conditions
\begin{equation}\label{eq:order-conditions}
  \sum_{r=1}^{R} c_r\,\rho_r \;=\; 1,
  \qquad
  \sum_{r=1}^{R} c_r\,\rho_r^{\,k} \;=\; 0
  \quad\text{for all odd } k,\ 3\le k\le p-1 .
\end{equation}
Then, for any $x_0\in\R^d,$ $\Delta_0>0,$ and for every $x\in\mathcal M(x_0,\Delta_0,\mu\rho_{\max})$, the
higher-order oracle \eqref{eq:ho-oracle} at $x$ satisfies
\[
  \gest^{(p)}(x) = \grad f(x) + \bias^{(p)}(x),
  \qquad
  \norm{\bias^{(p)}(x)} \;\le\; \kappa_p\,\mu^{\,p-1}\,M_{p}(x_0,\Delta_0),
  \qquad
  \kappa_p:=\frac{d}{p!}\sum_{r=1}^{R}\abs{c_r}\,\abs{\rho_r}^{p}.
\]
Moreover, $(c_r,\rho_r)_{r=1}^R$ satisfying \eqref{eq:order-conditions} always exists and $(c_r)_{r=1}^R$ is unique whenever $(\rho_r)_{r=1}^R$ is fixed.
\end{theorem}

\begin{proof}
Fix $x\in\mathcal M(x_0,\Delta_0,\mu\rho_{\max}).$ Let $M_k=M_k(x_0,\Delta_0,\mu\rho_{\max}),$ $u\in\sphere$, and set $g(s):=f(x+su)$, so that
$g^{(k)}(0)=\grad^k f(x)[u^{\otimes k}]$ and $\abs{g^{(k)}(s)}\le M_k$ for
$s\in[-\mu\rho_{\max},\mu\rho_{\max}]$ by Assumption~\ref{ass:smooth}. Taylor's
theorem with Lagrange remainder of order $p$ gives, for each $r$, some
$\theta_r\in(0,1)$ with
\[
  f(x+\mu\rho_r u)=g(\mu\rho_r)
  =\sum_{k=0}^{p-1}\frac{(\mu\rho_r)^k}{k!}\,g^{(k)}(0)
  +\frac{(\mu\rho_r)^{p}}{p!}\,g^{(p)}(\theta_r\mu\rho_r).
\]
Multiplying by $c_r$, summing over $r$, and dividing by $\mu$,
\[
  \frac1\mu\sum_{r=1}^R c_r\, f(x+\mu\rho_r u)
  =\sum_{k=0}^{p-1}\frac{\mu^{k-1}}{k!}\,S_k\,g^{(k)}(0)
  +\frac{\mu^{p-1}}{p!}\sum_{r=1}^R c_r\rho_r^{p}\,g^{(p)}(\theta_r\mu\rho_r),
  \qquad S_k:=\sum_{r=1}^R c_r\rho_r^{k}.
\]
Consider the radial moments $S_k$. By the assumed symmetry, each pair
$(\rho_r,c_r)$, $(-\rho_r,-c_r)$ contributes $c_r\rho_r^k\bigl(1-(-1)^k\bigr)$ to
$S_k$, which vanishes for even $k$. In particular $S_0=\sum_r c_r=0$, so the
singular $\mu^{-1}$ term and all even-order Taylor terms drop out. For odd $k$, the
order conditions \eqref{eq:order-conditions} give $S_1=1$ and $S_k=0$ for
$3\le k\le p-1$. Every term with $1\le k\le p-1$ therefore vanishes except $k=1$,
which contributes $g'(0)=\inner{\grad f(x)}{u}$, leaving
\[
  \frac1\mu\sum_{r=1}^R c_r\, f(x+\mu\rho_r u)
  =\inner{\grad f(x)}{u}
  +\frac{\mu^{p-1}}{p!}\sum_{r=1}^R c_r\rho_r^{p}\,g^{(p)}(\theta_r\mu\rho_r).
\]
Applying $\frac{d}{N}\sum_i(\cdot)\,u_i$, the linear part is reproduced exactly by Lemma~\ref{thm:exact}(i), so the residual is the
deterministic bias
\[
  \bias^{(p)}(x)
  =\frac{d}{N}\sum_{i=1}^N
  \frac{\mu^{p-1}}{p!}\sum_{r=1}^R c_r\rho_r^{p}\,
  \grad^p f(x+\theta_{r,i}\mu\rho_r u_i)[u_i^{\otimes p}]\,u_i,
  \qquad \theta_{r,i}\in(0,1).
\]
Each probe point lies in $\mathcal M(x_0,\Delta_0,\mu\rho_{\max})$ and $\norm{u_i}=1$, so
$\norm{\grad^p f(\cdot)[u_i^{\otimes p}]\,u_i}\le M_p$, and bounding term by term,
\[
  \norm{\bias^{(p)}(x)}
  \le \frac{d\,\mu^{p-1}M_p}{p!}\sum_{r=1}^R\abs{c_r}\,\abs{\rho_r}^{p}
  = \kappa_p\,\mu^{p-1}M_p.
\]

It remains to verify that the order conditions admit a unique set of antisymmetric weights for any choice of distinct positive nodes. Let $m:=\lfloor p/2\rfloor$, and denote the positive nodes by $\rho_1,\ldots,\rho_m$. By antisymmetry, the weights associated with the negative nodes are determined by those at the positive nodes. Setting $a_r:=c_r\rho_r$, the order conditions reduce to
$$
  2\sum_{r=1}^m a_r=1,
  \qquad
  2\sum_{r=1}^m a_r(\rho_r^2)^j=0,
  \quad j=1,\ldots,m-1.
$$
Thus, the $m$ unknowns $a_1,\ldots,a_m$ satisfy a Vandermonde system with nodes $\rho_1^2,\ldots,\rho_m^2$. Since the positive nodes are distinct, the squared nodes are also distinct, and hence the Vandermonde matrix is nonsingular. Therefore, the system admits a unique solution, and consequently the antisymmetric weights are uniquely determined.

\end{proof}

Theorem~\ref{thm:ho} shows that the augmented higher-order oracle reduces the bias to $\bigO(\mu^{p-1})$ (compared with the $\bigO(\mu^2)$ bias of the oracle \eqref{eq:base-oracle}), with the cost of $\bigO(d)$ per-step. For a $2$-design ($N=d+1$) and
fixed $p$, the query cost is $NR=\Theta(d)$. Moreover, Theorem~\ref{thm:ho} recovers Lemma~\ref{lem:bias}(i) when considering the oracle \eqref{eq:base-oracle} $p=3,\,\rho=\pm1,\,c=\pm\tfrac12$ with $c_3=d/6$. To conclude this section, we construct a new set of $(\rho_r,c_r)$ with $t=2,\,p=5,$ which will also be illustrated numerically in Section~\ref{ex:so}.

\begin{example}
Consider $t=2$ and $p=5$. We can construct the symmetric stencil as $(\rho_r,c_r)\in\{(\pm1,\pm\tfrac23),(\pm2,\mp\tfrac1{12})\}$. The resulting oracle
\eqref{eq:ho-oracle} uses $NR=4N$ evaluations per step, twice the $2N$ of
the  oracle \eqref{eq:base-oracle}, on the same $2$-design. In exchange,
Theorem~\ref{thm:ho} gives the bias bound $\tfrac{d}{18}\mu^4M_5(x_0,\Delta_0)$, of order
$\mu^4$, against the order $\mu^2$ of Lemma~\ref{lem:bias}(i).
\end{example}


\section{Convergence Analysis}\label{sec:conv}


In this section, we analyse the inheritance of local convergence
guarantees of FO methods by the design oracle ZO framework. Let $f$
belong to a structured class on which GD is known to
converge, for instance, a (strongly) convex, (strongly) quasar-convex, or
Polyak--{\L}ojasiewicz function, or, more generally, any class for which GD admits
a Lyapunov certificate. We show that the design-oracle method \eqref{eq:gd}, driven
by either the oracles in~\eqref{eq:base-oracle} or~\eqref{eq:ho-oracle}, inherits the FO decay behaviour and converges to a
neighbourhood of the GD fixed point, whose radius depends on the
biases obtained in Lemma~\ref{lem:bias} or Theorem~\ref{thm:ho}.


The ZO GD iterates \eqref{eq:gd}
can be viewed as the
discrete-time system
\begin{equation}\label{eq:gd-perturbed}
  x_{k+1}=x_k-\eta\bigl(\grad f(x_k)+\bias(x_k)\bigr),
\end{equation}
that is, gradient descent perturbed by the deterministic field $\bias(\cdot)$.

Let the FO gradient-descent map and its fixed point be
\begin{equation}\label{eq:gd-map}
  w(x):=x-\eta\,\grad f(x),
  \qquad
  z^e:\ w(z^e)=z^e\ \Longleftrightarrow\ \grad f(z^e)=0 .
\end{equation}
Let $x_0\in\R^d$ and $\Delta_0>0.$ For $x_k\in\mathcal M(x_0,\Delta_0)$, Lemma~\ref{thm:exact}(i) together with
Lemma~\ref{lem:bias}(i) or Theorem~\ref{thm:ho} gives the deterministic decomposition
$\gest^{(\cdot)}(x_k)=\grad f(x_k)+\bias(x_k)$, with
\begin{equation}\label{eq:struct-bias}
  \norm{\bias(x_k)}\le b,
  \qquad
  b:=\begin{cases}
     \tfrac{d}{6}\,\mu^{2}M_3(x_0,\Delta_0), & \text{for \eqref{eq:base-oracle}},\\[2pt]
     \kappa_p\,\mu^{p-1}M_p(x_0,\Delta_0), & \text{for \eqref{eq:ho-oracle}}.
     \end{cases}
\end{equation}
Substituting into \eqref{eq:gd} expresses the method as the deterministic perturbed
system
\begin{equation}\label{eq:perturbed-gd}
  x_{k+1}=w(x_k)+q_k,
  \qquad q_k:=-\eta\,\bias(x_k),
  \qquad \sup_{k}\norm{q_k}\le \eta\,b.
\end{equation}

\begin{theorem}
\label{thm:structured}
Let Assumptions~\ref{ass:smooth} and~\ref{ass:finite-M} hold with the order $p\geq3$ and
let $b$ be the corresponding bias bound \eqref{eq:struct-bias}. Assume GD map \eqref{eq:gd-map}  {generates iterates $\bar{x}_k$ and} has a fixed point $z^e$ that
is locally asymptotically stable and certified by a Lyapunov function $V$ on a ball
$B_R(z^e)$ as in Proposition~\ref{prop:iss-transfer},  {i.e., there exists a  class-$\mathcal{KL}$ function $\beta$ where $\|\bar{x}_k-z^e\|\leq \beta(\|\bar{x}_0-z^e\|,k)),\;k\geq0$}.  Let $\eta>0$
be the step size used by GD. Then
there exist $\bar\mu,\bar r>0$,
and a
class-$\mathcal K$ function $\gamma$ such that, for every $\mu\in(0,\bar\mu)$ and
every $x_0\in B_{\bar r}(z^e)$, the iterates of \eqref{eq:gd}, with $t\ge2,$
satisfy
\begin{equation}\label{eq:struct-iss}
  \norm{x_k-z^e}\;\le\;\beta\!\bigl(\norm{x_0-z^e},k\bigr)+\gamma(\eta\,b),
  \qquad k\ge 0 ,
\end{equation}
i.e., the design oracle method inherits the decay function $\beta$ of GD, and converges to a neighbourhood of $z^e$ of radius $\gamma(\eta\,b)$,
which tends to $0$ as $\mu\to0$. The iteration count matches that of
GD, and for $t\in\{2,3\}$ each iteration uses $\bigO(d)$ queries.
\end{theorem}

\begin{proof}
The iteration \eqref{eq:gd} reduces to the perturbed system \eqref{eq:perturbed-gd}, with $\sup_k\norm{q_k}\le\eta b$. 
By hypothesis, the unperturbed map $w$ satisfies the assumptions of
Proposition~\ref{prop:iss-transfer} on $B_R(z^e)$, with Lyapunov function $V$. The
perturbation in \eqref{eq:perturbed-gd} is bounded by $\eta b$, and $\eta b\to0$ as
$\mu\to0$. Choose $\bar\mu>0$ so that $\eta b\le\bar q$ for $\mu\in(0,\bar\mu)$,
where $\bar q$ and $\bar r$ are the admissible perturbation and initial condition
radii of Proposition~\ref{prop:iss-transfer}. Then, Proposition~\ref{prop:iss-transfer} yields
the estimate \eqref{eq:struct-iss}, with $\beta\in\mathcal{KL}$ the decay profile
of $w$ and $\gamma\in\mathcal K$. The exponential rate follows when the quadratic
Lyapunov hypothesis holds. The proposition also provides forward invariance of the
relevant sublevel set of $V$. Taking $R$ small enough that
$B_R(z^e)\subseteq\mathcal M(x_0,\Delta_0)$, the iterates remain in $\mathcal M(x_0,\Delta_0)$, so
\eqref{eq:struct-bias} is valid for all $k$ and the reduction above is justified.
Since $\beta$ is inherited unchanged, the iteration count matches that of GD, and
the per-iteration cost is that of the chosen oracle, which for $t\in\{2,3\}$ is $\bigO(d).$
\end{proof}

Theorem~\ref{thm:structured} strengthens the ISS reduction of
\cite{farzin2026cursed} in three respects.
First, \eqref{eq:perturbed-gd} holds for the trajectory itself, so no averaging is required. Second, The
perturbation bound $\eta b$ can be controlled through $\mu$ alone. Third, the reduction
holds under Assumptions~\ref{ass:smooth} and~\ref{ass:finite-M} in general, rather than only on the special
cases.

\begin{remark}[Momentum methods]\label{rem:momentum}
Because the reduction \eqref{eq:perturbed-gd} does not depend on the particular FO
map, the same argument applies to the heavy-ball and Nesterov accelerated methods
driven by the design oracle. Replacing the stochastic perturbation in the
augmented-state dynamics of \cite[Eqs.~(22)--(24), Thms.~4--5]{farzin2026cursed}
by the deterministic field $-\eta\,\bias(\cdot)$ of \eqref{eq:perturbed-gd} yields bounded perturbations under
Assumptions~\ref{ass:smooth} and~\ref{ass:finite-M}, and the same ISS conclusion follows.
\end{remark}

\section{Comparison with randomised oracles}
\label{sec:randomised-oracle-error}
Before going through the numerical results, we present
the estimation error of the Gaussian oracle of
\cite{nesterov2017random} and the unit-sphere oracle of \cite{flaxman2004online}, for our comparison with the spherical design methods.

\subsection{Error of the randomised oracle}
We compute the bounds of the bias, variance, and the root-mean-square error (RMSE) of the gradient estimation of the Gaussian and unit sphere oracles. The proofs of all results are deferred to Appendix~\ref{app:proof}.

\begin{proposition}[Bias and variance of the Gaussian oracle]
\label{prop:rand-error-G-i-ii}
Let Assumption~\ref{ass:smooth} hold with $p\ge3$, let Assumption~\ref{ass:finite-M}
be satisfied globally (i.e., $\mathcal M=\R^d$), and let
\[
  \gest(x)
  =
  \frac1m\sum_{j=1}^{m}
  \frac{f(x+\mu u_j)-f(x-\mu u_j)}{2\mu}\,u_j,
  \qquad
  u_j\sim\mathcal N(0,\Id),
\]
be the central-difference Gaussian oracle based on $m$ i.i.d.\ directions.
Write $\bar g(x):=\E[\gest(x)]$. Then, for $x\in\R^d,$ it holds that
\begin{alignat*}{3}
&(i) \text{\emph{ (Bias)}} && \quad  \bar g(x)=\grad f(x)+\bias(x),
  \qquad
  \norm{\bias(x)}
  \le
  \frac{\mu^2}{6}d(d+2)M_3, \\
&(ii) \text{\emph{ (Variance)}} && 
\quad    \E[\norm{\gest(x)-\bar g(x)}^2]
  \le
  \frac1m \!
  \left( \!
    2(d+2)\norm{\grad f(x)}^2
    \!+\!
    \frac{\mu^4}{18}
    d(d+2)(d+4)(d+6)M_3^2
 \! \right) \!.
\end{alignat*}
\end{proposition}

\begin{proposition}[Bias and variance of the unit sphere oracle]
\label{prop:rand-error-S-i-ii}
Let Assumption~\ref{ass:smooth} hold with $p\ge3$, let Assumption~\ref{ass:finite-M}
be satisfied, and let
\[
  \gest(x)
  =
  \frac{d}{m}\sum_{j=1}^{m}
  \frac{f(x+\mu u_j)-f(x-\mu u_j)}{2\mu}\,u_j,
  \qquad
  u_j\sim\unifS,
\]
be the central-difference spherical oracle based on $m$ i.i.d.\ directions.
Write $\bar g(x):=\E[\gest(x)]$ and let $M_3=M_3(x_0,\Delta_0)$. Then, for $x\in \mathcal M(x_0,\Delta_0),$ it holds that
\begin{alignat*}{3}
&(i) \text{\emph{ (Bias)}} && \qquad  \bar g(x)=\grad f(x)+\bias(x),
  \qquad
  \norm{\bias(x)}
  \le
  \frac{d\mu^2}{6}M_3, \\
&(ii) \text{\emph{ (Variance)}} && \qquad 
  \E[\norm{\gest(x)-\bar g(x)}^2]
  \le
  \frac1m
  \left(
    2d\norm{\grad f(x)}^2
    +
    \frac{d^2\mu^4}{18}M_3^2
  \right).
\end{alignat*}
\end{proposition}

\begin{proposition}[RMSE of the Gaussian oracle]
\label{prop:rand-error-G-iii-iv}
Under the setting of Proposition~\ref{prop:rand-error-G-i-ii}, and by using the notation $[z]_+:=\max\{z,0\}$, it holds that
\begin{alignat*}{3}
(i)& \quad  \E[\norm{\gest(x)-\grad f(x)}^2]^{1/2}
  \le
  \sqrt{\frac{2(d+2)}{m}}\,\norm{\grad f(x)}
  +
  \frac{\mu^2}{6}d(d+2)M_3
  +
  \mathcal O\!\left(
    \frac{d^2\mu^2M_3}{\sqrt m}
  \right), \\
(ii)&  \quad  \E[\norm{\gest(x)-\grad f(x)}^2]^{1/2}
  \ge
  \frac1{\sqrt m}
  \left[
    \sqrt{d+1}\,\norm{\grad f(x)}
    -
    \frac{\mu^2}{6}M_3
    \sqrt{d(d+2)(d+4)(d+6)}
  \right]_+.
\end{alignat*}
Moreover, if $\grad f(x)\neq0$ and
\(
  \mu^2
  \le
  \frac{3\sqrt{d+1}\,\norm{\grad f(x)}}
  {M_3\sqrt{d(d+2)(d+4)(d+6)}},
\)
then
\[
  \E[\norm{\gest(x)-\grad f(x)}^2]^{1/2}
  \ge
  \frac12\sqrt{\frac{d+1}{m}}\,
  \norm{\grad f(x)}.
\]
\end{proposition}

\begin{proposition}[RMSE of the unit sphere oracle]
\label{prop:rand-error-S-iii-iv}
Under the setting of Proposition~\ref{prop:rand-error-S-i-ii} and by using the notation $[z]_+:=\max\{z,0\}$,
it holds that
\begin{alignat*}{3}
(i)& \quad \E[\norm{\gest(x)-\grad f(x)}^2]^{1/2}
  \le
  \sqrt{\frac{2d}{m}}\,\norm{\grad f(x)}
  +
  \frac{d\mu^2}{6}M_3
  +
  \mathcal O\!\left(
    \frac{d\mu^2M_3}{\sqrt m}
  \right). \\
(ii)&  \quad   \E[\norm{\gest(x)-\grad f(x)}^2]^{1/2}
  \ge
  \frac1{\sqrt m}
  \left[
    \sqrt{d-1}\,\norm{\grad f(x)}
    -
    \frac{d\mu^2}{6}M_3
  \right]_+.
\end{alignat*}

Moreover, if $\grad f(x)\neq0$ and
\(
  \mu^2
  \le
  \frac{3\sqrt{d-1}\,\norm{\grad f(x)}}
  {dM_3},
\)
then
\[
  \E[\norm{\gest(x)-\grad f(x)}^2]^{1/2}
  \ge
  \frac12\sqrt{\frac{d-1}{m}}\,
  \norm{\grad f(x)}.
\]
\end{proposition}

The bounds of Propositions~\ref{prop:rand-error-G-iii-iv} and~\ref{prop:rand-error-S-iii-iv} show that, for both the Gaussian and the unit sphere randomised oracles, the estimation error contains a $\mu$-independent sampling-variance component of order $
  \sqrt{\frac{d}{m}}\,\norm{\grad f(x)}.
$
Consequently, decreasing the smoothing radius $\mu$ can reduce the finite-difference bias but it cannot eliminate the sampling error, which persists even as $\mu\to0$.

\begin{remark}
\label{rem:variance-floor-intrinsic}
The variance lower bounds in Propositions~\ref{prop:rand-error-G-iii-iv}
and~\ref{prop:rand-error-S-iii-iv} are not specific to Gaussian or spherical
sampling. More generally, the authors in \cite{gao2022generalizing} show that for directions
with i.i.d.\ zero-mean entries of variance $\Gamma^2$ and kurtosis $\kappa$,
the asymptotic MSE under the unbiased normalisation $\Gamma^2=1$ is
\[
  \E\norm{\gest(x)-\grad f(x)}^2
  \longrightarrow
  \frac{d+\kappa-2}{m}\norm{\grad f(x)}^2, 
\]
which is a $\mu$-independent term of order $\Theta(d/m)\,\norm{\grad f(x)}^2$ that no choice
of distribution removes. Thus, an $\Theta(d/m)$ sampling term persists as $\mu\to0$, with the minimum
attained by the Bernoulli distribution ($\kappa=1$).

\end{remark}

\begin{remark}\label{rem:design-removes}
The lower bounds in Propositions~\ref{prop:rand-error-G-iii-iv} and
\ref{prop:rand-error-S-iii-iv} both contain the sampling-variance term
$\Theta(\sqrt{d/m})\,\norm{\grad f(x)}$, so the two randomised oracles are
essentially equivalent in their dominant error. 
The spherical-design oracle removes the sampling-variance term 
entirely. 
\end{remark}

\subsection{The gradient estimation error}
\label{sec:accuracy}
 The next
result shows that the oracles will behave differently
as the target accuracy for the gradient estimation improves. The
design oracle attains any prescribed relative gradient accuracy at a query count independent of that accuracy, whereas the Gaussian sample cost grows with
$\delta^{-2}$.

\begin{theorem}\label{thm:accuracy}
Let $x_0\in\R^d$ and $\Delta_0>0.$ Fix $x\in\mathcal M(x_0,\Delta_0,\mu)$ with $\grad f(x)\neq 0$ and a target relative error
$\delta\in(0,1)$. Consider the oracles of
Propositions~\ref{prop:rand-error-G-i-ii} and~\ref{prop:rand-error-S-i-ii}, the Gaussian and unit sphere oracle with $m$ samples, and the
base $2$-design oracle \eqref{eq:base-oracle} with $N=d+1$.
\begin{enumerate}[label=(\roman*)]
\item \emph{(Gaussian and unit sphere oracles)} The relative RMSE obeys
\[
  \frac{\E[\norm{\gest(x)-\grad f(x)}^2]^{1/2}}{\norm{\grad f(x)}}
   \ge \bigOm \left(\sqrt{\frac{d}{m}}\right)\
  \quad \text{for} \quad \mu \to 0.
\]
Moreover, attaining a relative error of at most $\delta$ requires
$m\ge(d+1)/\delta^{2}=\Theta(d\,\delta^{-2})$ samples.
\item \emph{(Design oracle)} 
Any
\(
  \mu\le\mu_\delta=\sqrt{\frac{6\,\delta\,\norm{\grad f(x)}}{d\,M_3(x_0,\Delta_0,\mu)}}
\)
yields a relative RMSE of at most $\delta,$ using $N=d+1$ directions  (independent of $\delta$).
\end{enumerate}
\end{theorem}

\begin{proof}
By Propositions~\ref{prop:rand-error-G-iii-iv} and~\ref{prop:rand-error-S-iii-iv} with sufficiently small $\mu,$ and by upper bounding the error by $\delta,$ item (i) can be concluded.
Considering Lemma~\ref{thm:exact}(i),  Lemma~\ref{lem:bias}(i) and imposing $\norm{\bias(x)}\le\tfrac{d}{6}\mu^2 M_3(x_0,\Delta_0)\le\delta\norm{\grad f(x)}$ gives
$\mu\le\mu_\delta$, i.e., item (ii) follows. We additionally observe that the direction count $N=d+1$ does not depend on $\delta$.
\end{proof}

Considering Theorem~\ref{thm:accuracy}, both the design and randomised oracles cost $\bigO(d)$ per gradient at a fixed accuracy. When
$\delta\to0$ the randomised oracles' cost grows to $\Theta(d\,\delta^{-2})$, whereas the
$2$-design cost remains $\Theta(d)$. The design oracle is therefore strictly 
cheaper  for every $\delta<1$, and the gap widens by the factor $\delta^{-2}$ as $\delta\to0$.

\section{Numerical Examples}\label{sec:numerics}


The experiments are implemented in Python and run on a laptop with an Intel Core
i7-1265U processor.\footnote{The code is available at \url{https://...}.} Unless
stated otherwise we work in dimension $d=8$, where an exact design is available in
closed form. For the oracle \eqref{eq:base-oracle} we use the regular simplex,
an exact $2$-design of $N=d+1=9$ points. For the $4$-design oracle we use the
$E_8$ lattice set ($N=240$) and the higher-order oracle \eqref{eq:ho-oracle} is taken at
$p=5$, that is, the simplex $2$-design augmented with the symmetric radial stencil
$\rho\in\{\pm1,\pm2\}$, $c\in\{\pm\tfrac23,\mp\tfrac1{12}\}$ obtained from \eqref{eq:order-conditions}, whose bias is $\bigO(\mu^{4})$. 
We compare our results with the randomised baselines, including Gaussian
\citep{nesterov2017random} and spherical oracles \citep{flaxman2004online}, whose directions are sampled from the Gaussian distribution and the uniform distribution over the unit sphere. Both of the ZO oracles are evaluated with $m=d+1$ central-difference directions, matching the per-step evaluation budget $2(d+1)$ of the base design oracle. Their curves are averaged over independent seeds, with shaded bands indicating $\pm1$ standard deviation. 

For a differentiable $f$, the estimation error of a finite-difference oracle consists of two terms:
\begin{equation}\label{eq:num-errsplit}
  \|\gest(x)-\grad f(x)\|
  \;\approx\;
  \underbrace{C_{\mathrm{trunc}}\,\mu^{\,p-1}}_{\text{smoothing bias}}
  \;+\;
  \underbrace{C_{\mathrm{round}}\,\frac{\varepsilon_{\mathrm{mach}}\,|f(x)|}{\mu}}_{\text{roundoff}}.
\end{equation}
The second term arises from the floating-point cancellation incurred when the
nearly equal values are subtracted and divided by $2\mu$. 
When the oracle reconstructs the gradient exactly, only the roundoff term remains and the error curve would decrease with $\mu.$ Such a phenomenon will be observed in some of the figures below.


\subsection{Convergence on a quartic objective}\label{ex:so}
Consider the optimization problem
\begin{equation}\label{eq:quartic}
    \min_{x\in\R^d} \quad \left\{f(x)=\frac12 x^\top A x+\frac{\alpha}{4}\sum_{i=1}^d (x_i-1)^4 \right\}
\end{equation}
with $\alpha\in\{0,0.5,2\}.$  When $\alpha=0,$ it is a quadratic problem with condition number $\kappa=3;$ and when $\alpha>0,$ it is a quartic problem with non-vanishing $\grad^3 f$ at the optimum. We generate the matrix $A=Q\,\mathrm{diag}(\lambda_1,\dots,\lambda_d)\,Q^\top,$ where $Q$ is a random $d\times d$ orthogonal matrix and $\lambda_1,\dots,\lambda_d$ logarithmically spaced between $1$ and $\kappa$, so that $A$ is positive definite with condition number $\kappa$.
We use the step size $\eta=c/L$ with $c=1$ for
the quadratic and $c=\tfrac12$ for the quartics (which satisfy the step size condition in our theory), a fixed smoothing radius $\mu=10^{-2}$, and run $200$ iterations. Randomised curves are averaged over $50$ seeds with $\pm1$ standard-deviation bands shaded in the figure. 

\begin{figure}[tbp]
  \centering
  \includegraphics[width=0.9\textwidth]{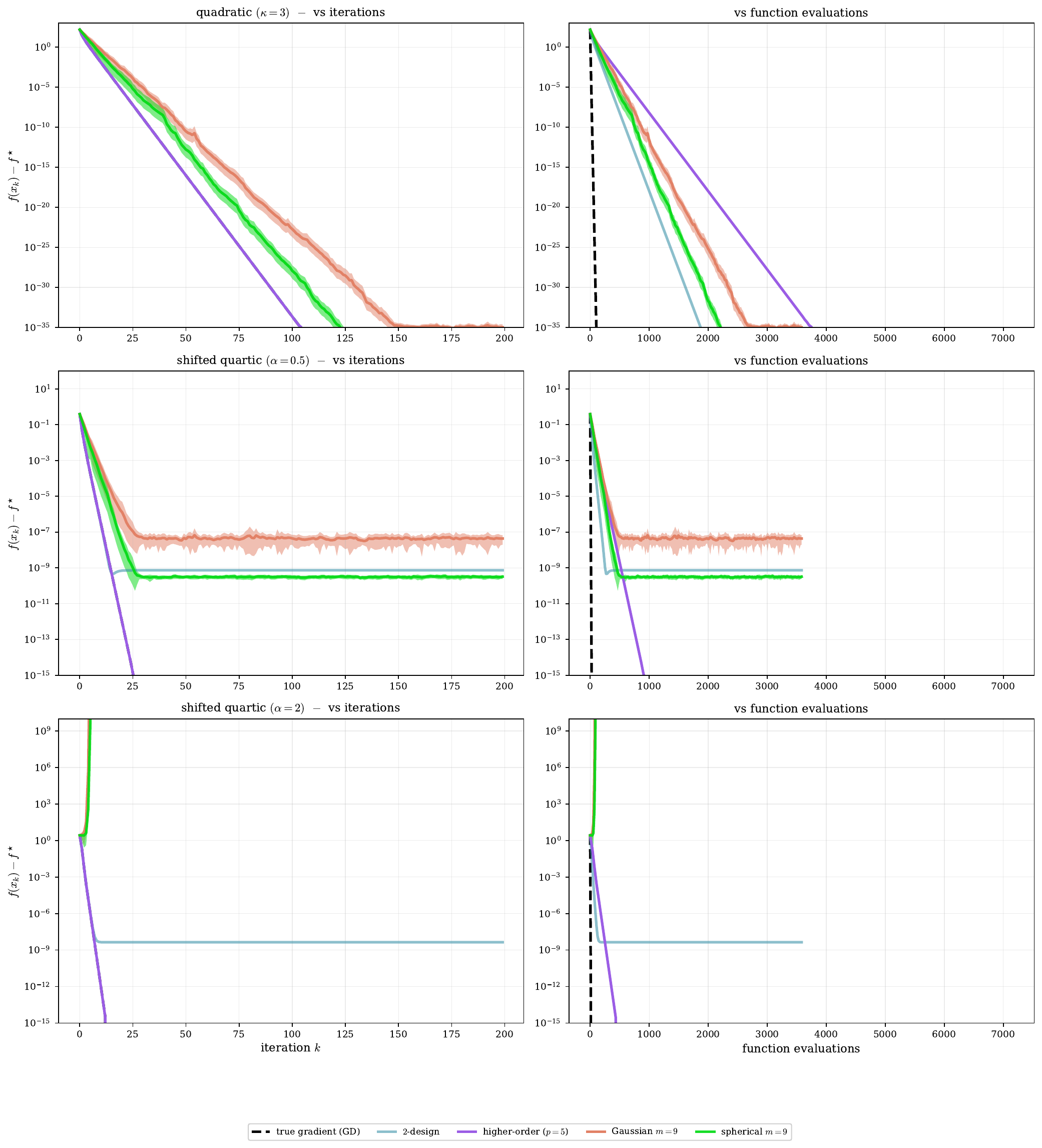}
  \caption{Optimality gap $f(x_k)-f^\star$ versus iterations (left column) and
  function evaluations (right column), averaged over $50$ seeds with $\pm1$
  standard-deviation bands. The top row considers the quadratic setting ($\alpha=0$); the middle and the bottom rows consider the quartic setting with $\alpha=0.5$ and $\alpha=2,$ respectively.}
  \label{fig:convergence}
\end{figure}

Figure~\ref{fig:convergence} reports the optimality gap
against both the iteration count and the number of function evaluations. We see that, when $\alpha=0$ (top row), the optimality gap $f(x_k) - f^\ast$ of the iterates of the 2-design oracle coincides with that of the true
gradient descent, verifying our result in Lemma~\ref{thm:exact}.
The randomised Gaussian and spherical
oracles converge
but their trajectories fluctuate by about one
standard deviation around the mean. 
Since their variance is proportional to
$\|\grad f(x_k)\|^2$ (see Propositions~\ref{prop:rand-error-G-iii-iv} and~\ref{prop:rand-error-S-iii-iv}), it decreases as the iterates get closer to the optimum. Consequently, the randomised oracles still converge but at a slower rate.

When $\alpha=0.5$ (the middle row), we see that all methods, except the higher-order oracle, converges
linearly and then settles at a plateau. This matches our results in Theorem~\ref{thm:structured}, where the iterates are only guaranteed to converge
to a neighbourhood of the equilibrium. Among all, the Gaussian oracle gives the largest objective value, as verified by the bound in Propositions~\ref{prop:rand-error-G-i-ii} and~\ref{prop:rand-error-S-i-ii}.
When $\alpha=2$ (the bottom row), the Gaussian and
spherical oracles diverge, whereas the design oracles remain stable. the third-order derivative $M_3(x_0,\Delta_0,\mu)$ is larger, so is the radius of the converging neighbourhood.
For both settings ($\alpha=0.5$ and $\alpha=2$), the higher-order
oracle tracks the true
gradient descent to machine precision.

\subsection{Gradient-estimation error}
Next, we examine the gradient estimation error of each method at a fixed point that is randomly drawn from the normal distribution. We consider the optimisation problem~\eqref{eq:quartic} and again, we consider the quadratic objective ($\alpha=0$) and the quartic objective ($\alpha=1$). For the quadratic objective, we set $\kappa=20;$ and for the quartic objective, we set $\kappa=5.$ We sweep the smoothing radius $\mu$
and record $\|\gest(x)-\grad f(x)\|$ for the
deterministic oracles and the RMSE over $50$ seeds for the
randomised oracles. Figure~\ref{fig:est-error} shows the gradient estimation error against different smoothing radius $\mu$ at the fixed point.
\begin{figure}[tbp]
  \centering
  \includegraphics[width=0.85\textwidth]{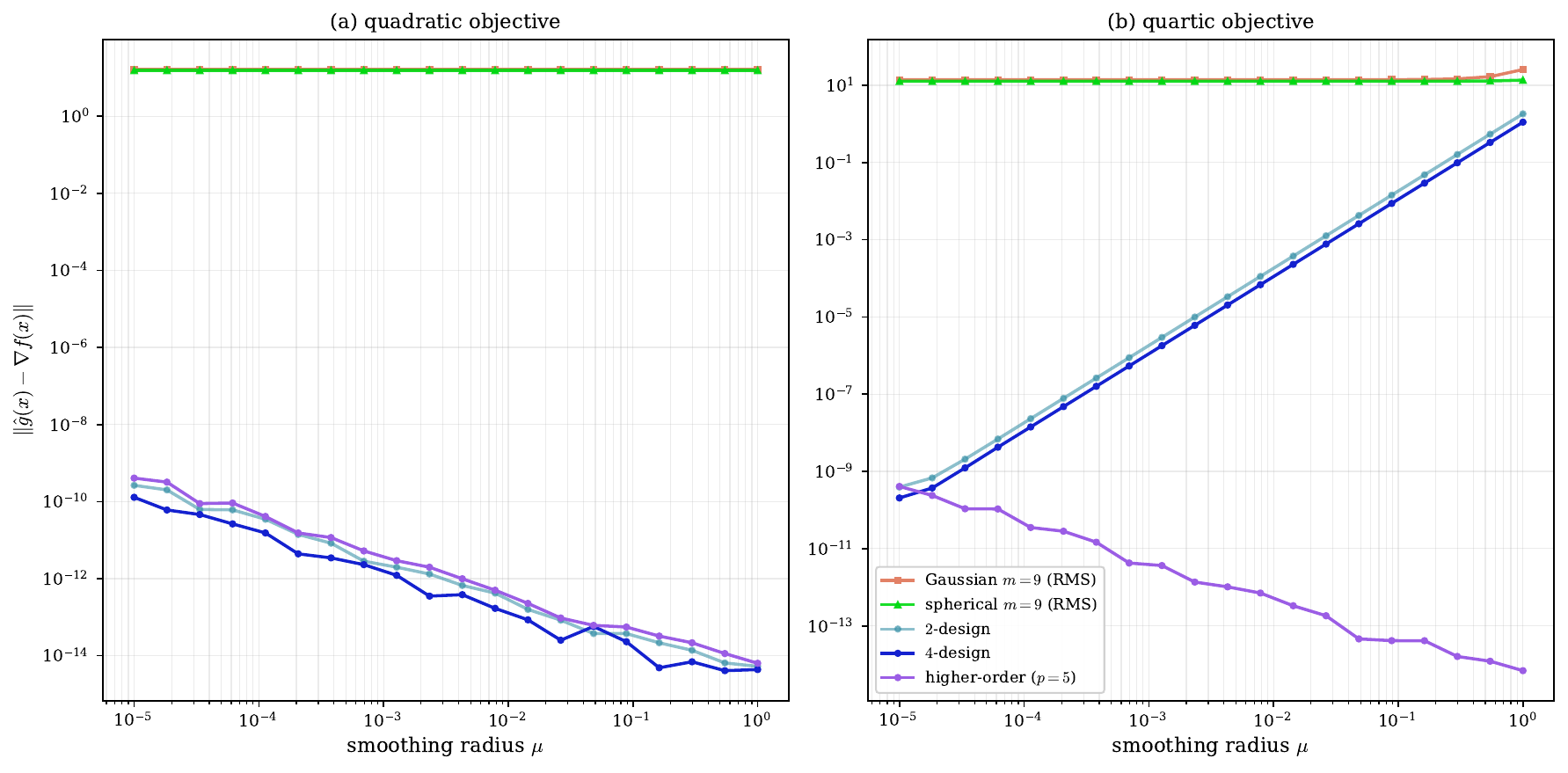}
  \caption{Gradient-estimation error versus smoothing radius $\mu$ at a fixed
  point ($d=8$): \textbf{(a)} a quadratic objective; \textbf{(b)} a quartic
  objective.}
  \label{fig:est-error}
\end{figure}
Figure~\ref{fig:est-error}(a) shows the error for the quadratic objective. From Lemma~\ref{thm:exact}, we know that a $2$-design reproduces the gradient of any affine or quadratic function exactly (i.e., no bias), so only the roundoff term in \eqref{eq:num-errsplit} remains in the residual, which decreases as $\mu$ increases. On the contrary, the Gaussian and unit sphere estimators are flat independent of the smoothing radius $\mu,$ matching the variance lower bound $\Theta(\sqrt{d/m})\|\grad f(x)\|$ in Propositions~\ref{prop:rand-error-G-iii-iv} and~\ref{prop:rand-error-S-iii-iv}.

Figure~\ref{fig:est-error}(b) shows the error for the quartic objective. The $2$-design and the $4$-design both demonstrate the $\bigO(\mu^2)$ bias as proved in Lemma~\ref{lem:bias}. 
The higher-order oracle outperforms all other methods. With order $p=5,$ the bias bound in Theorem~\ref{thm:ho} vanishes and it can well approximate the gradient of a quartic function.

\subsection{Non-differentiable objective}
\label{sec:numerics-nonsmooth}
In this experiment, we consider a non-smooth objective, which violates the continuously differentiability assumption in Assumption~\ref{ass:smooth}. We consider the
convex max-of-quadratics
\[
  f(x)=\max_{1\le j\le J}\Bigl(\tfrac12 x^\top A_j x + b_j^\top x\Bigr),
  \qquad A_j \succeq 0,
\]
with dimension $d=8$ and $J=6.$ To generate $f$, we fix a seed, sample $\tilde b_1,\dots,\tilde b_J \sim \mathcal{N}(0,I_d)$ independently, and normalise $b_j$ such that
$\sum_{j=1}^J b_j=0$. Each matrix is generated as
\(
A_j=\frac1d B_jB_j^\top+\gamma I_d
\)
with $\gamma=0.5$, where $B_j\in\R^{d\times d}$ has i.i.d.\ standard normal entries. 
Given positive definite matrices $A_j$ ($j=1,\ldots,J$) and $\sum_{j=1}^J b_j=0,$ we know that
\[
f(x)\ \ge\ \frac1J\sum_{j=1}^J\Bigl(\tfrac12 x^\top A_j x+b_j^\top x\Bigr)\ \ge\ \frac{\gamma}{2}\|x\|^2 \qquad \text{for all } x\in\R^d,
\]
and thus $f$ is minimised at $x^\star=0$ with
\(
\partial f(0)=\mathrm{conv}\{b_1,\dots,b_J\}
\)
and $f^\star=f(0)=0.$

We compare the performance of the oracle \eqref{eq:base-oracle} with $2$-design,
the higher-order oracle with $p=5$, with that of the randomised Gaussian and spherical oracles
($m=d+1$), and the classical subgradient method. Consider $\bar L:=\max_j\lambda_{\max}(A_j).$ Figure~\ref{fig:nonsmooth} reports the optimality gap against the number of iterations when a constant step size $\bar\eta=c/\bar L$ with $c\in\{0.1,1\}$ and the smoothing radius $\mu\in\{10^{-2},10^{-4}\}$ are used for all methods.

\begin{figure}[tbp]
  \centering
  \includegraphics[width=0.9\textwidth]{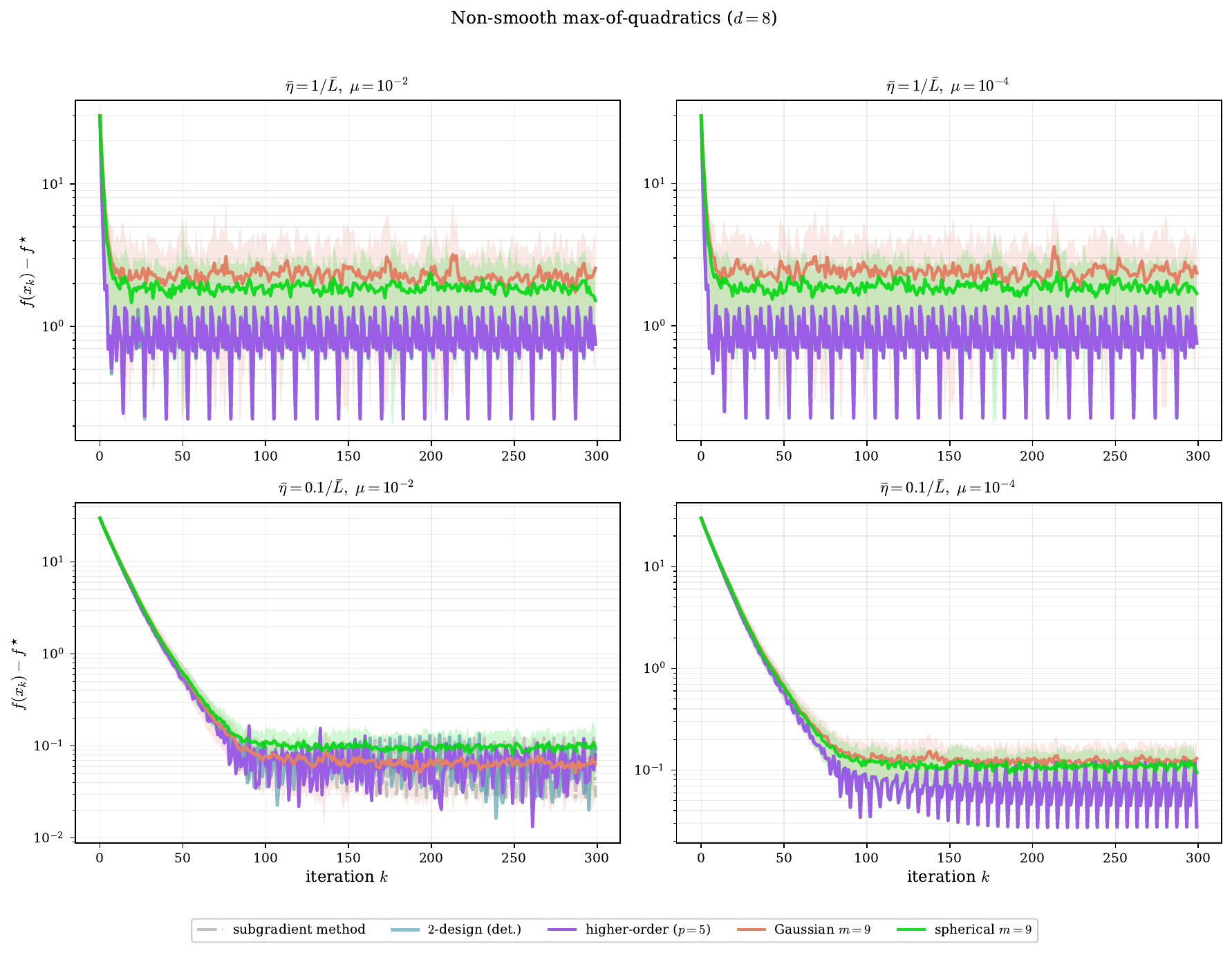}
  \caption{Non-smooth max-of-quadratics: optimality gap
  versus iteration under the constant step size $\bar\eta=c/\bar L$, with $c\in\{0.1,1\}$ and $\mu\in\{10^{-2},10^{-4}\}$.}
  \label{fig:nonsmooth}
\end{figure}


We see that, in this example, the design oracles remain stable on the
non-smooth problem and behaves comparably to the subgradient method. 
With a larger step size (i.e., $\bar{\eta}=1/\bar{L}$), all methods oscillate in a band around $f^\star$;
and with a smaller step size (i.e., $\bar{\eta}=0.1/\bar{L}$), every method descends to a
common $\bigO(10^{-1})$ neighbourhood. The smoothing radius $\mu\in\{10^{-2},10^{-4}\}$ does not affect much the performance of all methods, just with a slight improvement of the $2$-design and higher-order design when a smaller smoothing radius and smaller step size are used.

Moreover, we see that the higher-order oracle offers \emph{no} advantage
over the base $2$-design when $f$ is non-smooth. 
Under this setting, the Taylor terms in Theorem~\ref{thm:ho} cannot be cancelled and thus the $\bigO(\mu^{p-1})$ bias of the higher-order oracle does not apply. This shows that the order of continuous differentiability $p$ is essential for the higher-order oracle. When the smoothing radius is larger (see the bottom left figure in Figure~\ref{fig:nonsmooth}), the iterate crosses more active pieces of the max near $x^\star$, resulting in a larger oscillation of the objective value. 

\subsection{Neural-network classification}
\label{sec:numerics-nn}
Lastly, we consider a higher-dimensional example, where the per-step cost of an exact design (specifically, $2(d+1)$ evaluations) becomes substantial. The spherical
design oracle, exact up to the smoothing bias, can apply a large step $\eta$, whereas
the randomised oracles, whose variance of the estimate is given by $\Theta(d/m)$ in
Propositions~\ref{prop:rand-error-G-iii-iv} and~\ref{prop:rand-error-S-iii-iv}, can only perform a small step $\eta_{\mathrm{rand}}$ for stability. The question is: Which is the more budget-friendly option with limited function evaluations?

To address this, we train a one-hidden-layer neural network for binary classification. For an input
$\xi\in\R^P$ the network computes
\[
  \widehat y(\xi;\theta)
  = \Gamma\!\bigl(w_2^\top \zeta(W_1\xi+b_1)+b_2\bigr),
  \qquad
  \zeta(z)=\text{ReLU}(z),
  \quad
  \Gamma(z)=\tfrac{1}{1+e^{-z}},
\]
with hidden width $H=64$, ReLU activation $\zeta$ applied componentwise, and parameters $\theta=(W_1,b_1,w_2,b_2)$ stacked into a single vector of
dimension $d=HP+2H+1$.  
The loss is the $\ell_2$-regularised binary
cross-entropy over the $n$ training examples $(\xi_i,y_i)$,
\[
  f(\theta)=\frac1n\sum_{i=1}^{n}
  \Bigl[-y_i\log\widehat y(\xi_i;\theta)-(1-y_i)\log\bigl(1-\widehat y(\xi_i;\theta)\bigr)\Bigr]
  +\tfrac{\lambda}{2}\norm{\theta}^2
\]
with the regularisation parameter $\lambda=10^{-4}.$ Each ZO query is one
forward pass over the training set, and we report progress against the number of queries, the operative cost in the derivative-free setting.

We use two standard UCI binary-classification datasets, chosen to place the
parameter dimension in the target range. The first is the Wisconsin Diagnostic
Breast Cancer set ($P=30$ features, $569$ examples), giving
$d=30\cdot64+2\cdot64+1=2049$. The second is the Spambase set ($P=57$ features),
subsampled to $1000$ examples, giving $d=57\cdot64+2\cdot64+1=3777$. At these
dimensions a full simplex $2$-design uses $N=d+1$ directions, so a single design
step costs $2(d+1)\approx4100$ and $7560$ evaluations respectively,
orders of magnitude more than a randomised step with $m=10$ directions. Each dataset
is split $75/25$ into training and test sets by a seeded permutation and
standardised using the training-set statistics. The weights are initialised at
$1/\sqrt{P}$- and $1/\sqrt{H}$-scaled random values and the biases at zero.

We compare the base $2$-design oracle, using the full simplex design
($N=d+1$ directions, $2(d+1)$ evaluations per step), against the Gaussian and
spherical oracles with $m=10$ directions ($2m=20$ evaluations per step). The
smoothing radius is $\mu=0.1$ throughout. To set the step size, we estimate the gradient Lipschitz constant $\bar L$ numerically and use $\eta=1/\bar L$ for the design oracle and $\eta_{\mathrm{rand}}=\eta\,m/d$ for the randomised oracles, which is consistent with the $\Theta(d/m)$ variance in Propositions~\ref{prop:rand-error-G-iii-iv} and~\ref{prop:rand-error-S-iii-iv}.

Figure~\ref{fig:nn-bc} reports the training loss and the test and training accuracy
against the number of function evaluations on the breast-cancer dataset. We see that all methods reach comparable final accuracy, showing that the smoothed ZO oracles optimise the network effectively. Moreover, even though the per-step cost of a 2-design is higher, it attains a lower training loss than the randomised oracles at equal budget with a larger step size.
\begin{figure}[tbp]
  \centering
  \includegraphics[width=\textwidth]{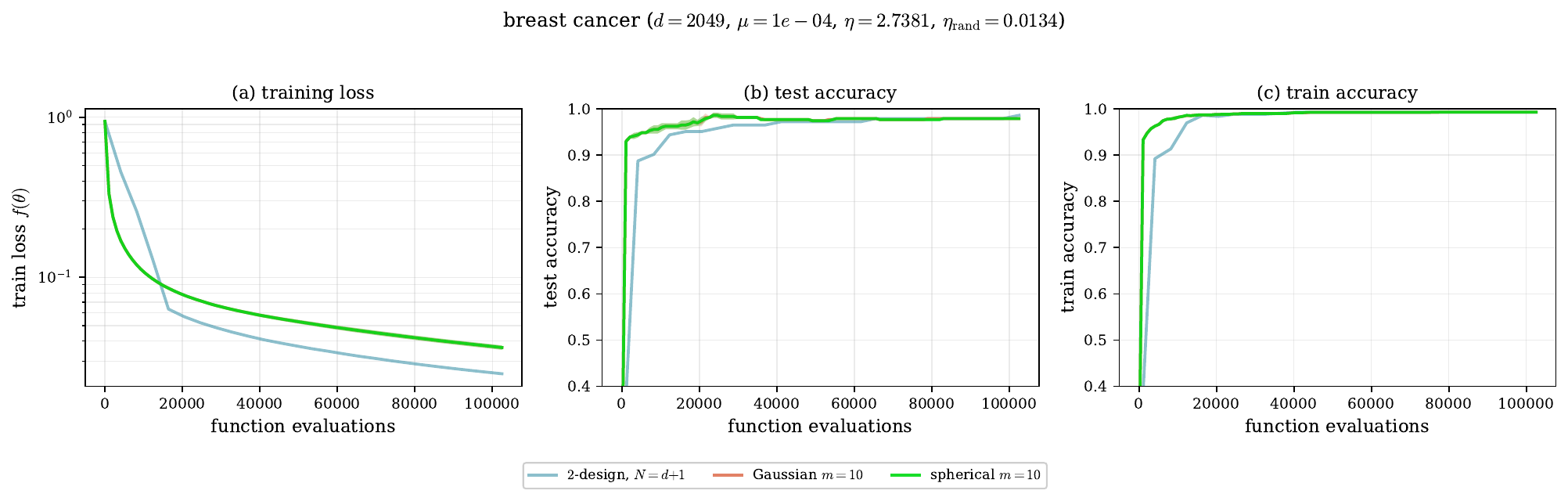}
  \caption{Breast-cancer classifier,
  $c=1$): \textbf{(a)} training loss, \textbf{(b)} test accuracy, and
  \textbf{(c)} training accuracy versus the number of function evaluations.}
  \label{fig:nn-bc}
\end{figure}

Similarly, Figure~\ref{fig:nn-spam} shows the training loss and the test and training accuracy
against the number of function evaluations on the Spambase dataset at a higher dimension
$d=3777.$ 
As can be seen, the $2$-design obtains a lower training loss and comparable accuracy curves compared with the randomised ZO oracles.
\begin{figure}[tbp]
  \centering
  \includegraphics[width=\textwidth]{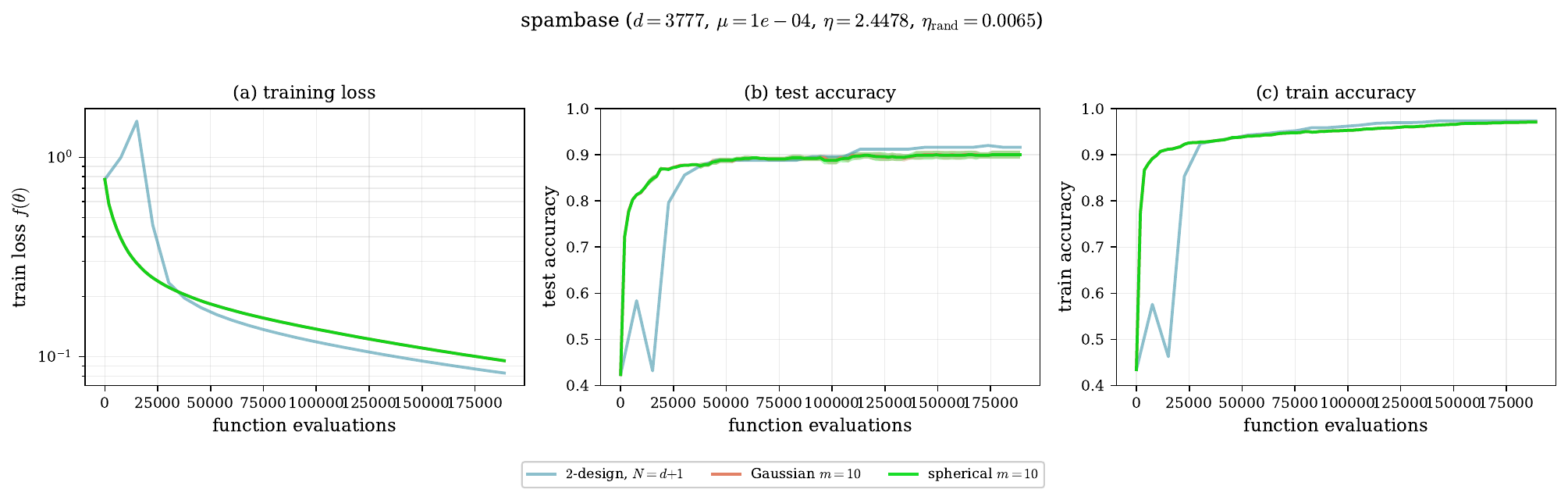}\\[4pt]
  \includegraphics[width=\textwidth]{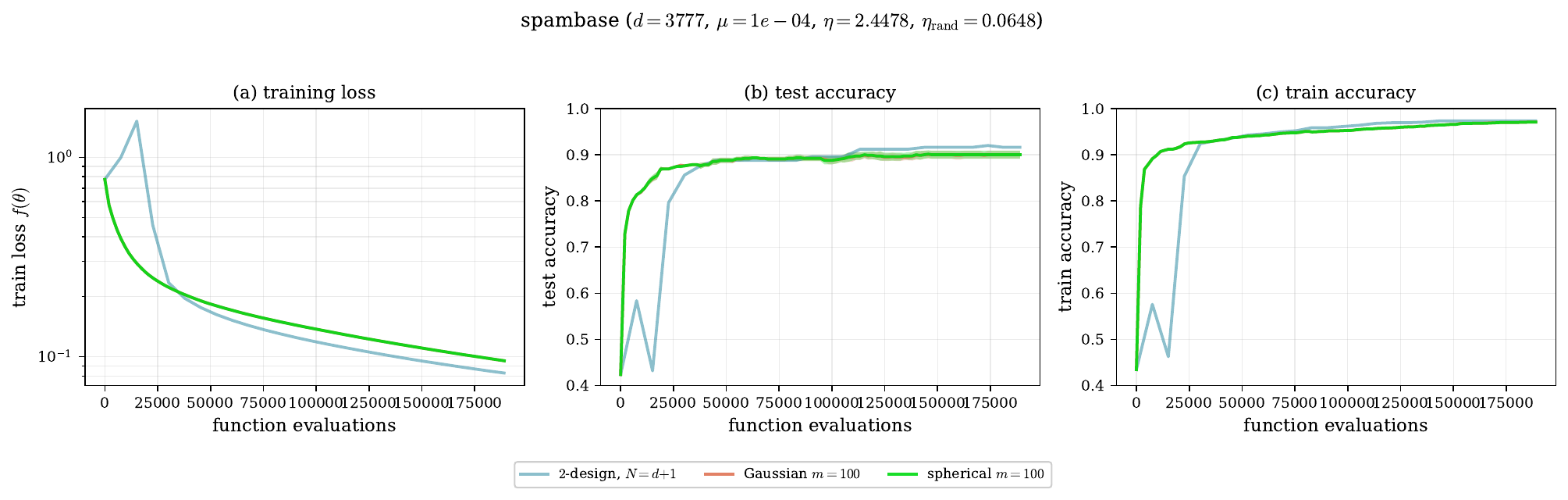}
  \caption{Spambase classifier: training loss and
  test/training accuracy versus function evaluations.}
  \label{fig:nn-spam}
\end{figure}

\section{Conclusion and discussion}\label{sec:discussion}
In this work, we proposed the ZO spherical design oracle. We introduced the concept of spherical design and showed its bias when applying the central differences. We further proposed a general finite differences method that captures the higher-order smoothness of the objective and showed its bias. We compared our methods with the existing randomised ZO methods by developing the bias, variance and the RMSE of the existing methods. We showed that our proposed method outperformed the existing randomised methods in several settings.
One possible future direction is to study the noisy function evaluation setting. Our current the analysis is noiseless and thus eliminates the sampling variance of the
direction set. Thus, it is not clear how the noise of the function evaluation would affect our proposed ZO method.
Also, since the ZO spherical design oracle needs
$\bigOm(d^{\lfloor t/2\rfloor})$ of query cost, it would be interesting to study the application of spherical design in high-dimensional settings.
Another direction is to analyse the guarantees and performance of the design oracle in solving the non-differentiable optimisation problems.

\backmatter

\begin{appendices}

\section{Proofs of Section~\ref{sec:randomised-oracle-error}}\label{app:proof}
Here, we present the proofs of propositions in Section~\ref{sec:randomised-oracle-error}.
To shorten expressions in the following, we define the central-difference quantity
\(
  D_\mu(u):=\frac{f(x+\mu u)-f(x-\mu u)}{2\mu}.
\)
For $g(s):=f(x+su)$, Taylor's theorem and 
the intermediate value theorem give 
\begin{equation}\label{eq:cdiff-rem}
  D_\mu(u)
  =
  \inner{\grad f(x)}{u}
  +\frac{\mu^2}{6}\,
    \grad^3 f(x+\zeta u)[u,u,u].
\end{equation}
for some $\zeta\in(-\mu,\mu)$.
Here, the even-order terms cancel by antisymmetry. From the definition of $M_3$ in \eqref{eq:Mj} we obtain the bound 
\begin{equation}\label{eq:cdiff-error}
  \abs{D_\mu(u)-\inner{\grad f(x)}{u}}
  \le
  \frac{\mu^2}{6}M_3\norm{u}^3
\end{equation}
used in the following proofs.

Moreover, the proofs for the Gaussian oracle use the following moment identities for
$u\sim\mathcal N(0,\Id)$. Since $\norm u^2$ has a chi-square distribution
with dimension $d$ by assumption, $\E[\norm u^{2k}]=\prod_{i=0}^{k-1}(d+2i)$
\citep[Ch.~18, eq (18.8)]{johnson1994continuous}, and in particular
\begin{equation}\label{eq:gauss-norm-moments}
  \E[\norm u^{4}]=d(d+2)
  \quad \text{and} \quad
  \E[\norm u^{8}]=d(d+2)(d+4)(d+6).
\end{equation}
Similarly, Isserlis' theorem \citep{isserlis1918formula} gives
$\E[u_au_bu_cu_e]=\delta_{ab}\delta_{ce}+\delta_{ac}\delta_{be}
+\delta_{ae}\delta_{bc}$, the Gaussian counterpart of \eqref{eq:moment4},
and which allows us to use the qualities 
\begin{equation}\label{eq:gauss-inner-moments}
  \E[\inner au^2]=\norm a^2,
  \qquad
  \E[\inner au^2\norm u^2]=(d+2)\norm a^2 
\end{equation}
for $a\in\R^d$. For $u\sim\unifS$ the corresponding identity is \eqref{eq:moment2}, which
gives $\E[\inner au^2]=a^\top\E[uu^\top]a=\tfrac1d\norm a^2$.

\begin{proof}[Proof of Proposition~\ref{prop:rand-error-G-i-ii}]
Let
\(
  Y_j:=D_\mu(u_j)u_j \)
  and \(  \gest(x)=\frac1m\sum_{j=1}^m Y_j.
\)

\emph{(i)} 
Since $\E[uu^\top]=\Id$, \eqref{eq:cdiff-rem} gives
\(
  \bar g(x)
  =
  \grad f(x)
  +
  \frac{\mu^2}{6}
  \E\!\left[
    \grad^3f(x+\zeta u)[u,u,u]u
  \right].
\)
Thus, using
$\norm{\grad^3f(\cdot)[u,u,u]u}\le M_3\norm u^4$ (where $M_3$ is defined in \eqref{eq:Mj}) and
$\E\norm u^4=d(d+2),$ which follows from \eqref{eq:gauss-norm-moments},
\[
  \norm{\bias(x)}
  \le
  \frac{\mu^2}{6}M_3\E\norm u^4
  =
  \frac{\mu^2}{6}d(d+2)M_3.
\]

\emph{(ii)} 
Since $u_1,\dots,u_m$ are i.i.d., so are $Y_1,\dots,Y_m$, and hence
\begin{align}
  \E[\norm{\gest(x)-\bar g(x)}^2]
  =
  \frac1m\Var(Y_1)
  \le
  \frac1m\E\norm{Y_1}^2
  =
  \frac1m\E[D_\mu(u)^2\norm u^2]. \label{eq:aux_iid}
\end{align}
Using \eqref{eq:cdiff-rem} and $(a+b)^2\le2a^2+2b^2$, leads to the inequality
\[
  \E[D_\mu(u)^2\norm u^2]
  \le
  2\E[\inner{\grad f(x)}u^2\norm u^2]
  +
  \frac{\mu^4}{18}M_3^2\E\norm u^8.
\]
Then, the Gaussian moment identities
\[
  \E[\inner au^2\norm u^2]=(d+2)\norm a^2,
  \qquad
  \E\norm u^8=d(d+2)(d+4)(d+6)
\]
which follow from \eqref{eq:gauss-norm-moments} and \eqref{eq:gauss-inner-moments},
yield the result.
\end{proof}

\begin{proof}[Proof of Proposition~\ref{prop:rand-error-S-i-ii}]
Let $M_3=M_3(x_0,\Delta_0),$ 
  $Y_j:=D_\mu(u_j)u_j$,
  and
  $\gest(x)=\frac{d}{m}\sum_{j=1}^m Y_j.$

\emph{(i)}
From \eqref{eq:moment2} we 
know that $\E[uu^\top]=\frac1d\Id$. Hence, by
\eqref{eq:cdiff-rem}, $\bar g(\cdot)$ can be written as
\(
  \bar g(x)
  =
  \grad f(x)
  +
  \frac{d\mu^2}{6}
  \E\!\left[
    \grad^3f(x+\zeta u)[u,u,u]u
  \right].
\)
Since $\norm u=1$, the estimate
\(
  \norm{\bias(x)}
  \le
  \frac{d\mu^2}{6}M_3
\)
follows from the definition of $M_3$ in \eqref{eq:Mj}.

\emph{(ii)}
Since $u_1,\dots,u_m$ are i.i.d., so are $Y_1,\dots,Y_m$. As in \eqref{eq:aux_iid} we can conclude that
\(
  \E[\norm{\gest(x)-\bar g(x)}^2]
  =
  \frac{d^2}{m}\Var(Y_1)
  \le
  \frac{d^2}{m}\E[D_\mu(u)^2\norm u^2]
\)
and by using 
\eqref{eq:cdiff-rem} the inequality
\[
  \E[D_\mu(u)^2\norm u^2]
  \le
  2\E[\inner{\grad f(x)}u^2]
  +
  \frac{\mu^4}{18}M_3^2
\]
follows. Finally, the variance bound follows from the identity
$\E[\inner au^2]=\frac1d\norm a^2$ 
(see \eqref{eq:moment2}).
\end{proof}

Before we continue with a proof of Proposition~\ref{prop:rand-error-G-iii-iv}, for
the RMSE, we introduce 
the following common notation 
\[
  a:=\grad f(x),
  \qquad
  A:=\inner au\,u,
  \qquad
  B:=\bigl(D_\mu(u)-\inner au\bigr)u, \qquad Y=A+B.
\]
Additionally, we note that by
the reverse triangle inequality for the $L^2$ norm it holds that
\begin{equation}\label{eq:variance-perturbation}
  \sqrt{\Var(Y)}
  \ge
  \sqrt{\Var(A)}-\sqrt{\Var(B)}
\end{equation}
and 
\eqref{eq:cdiff-error} implies
\begin{equation}\label{eq:B-variance-bound}
  \sqrt{\Var(B)}
  \le
  \sqrt{\E\norm B^2}
  \le
  \frac{\mu^2M_3}{6}\sqrt{\E\norm u^8}.
\end{equation}

\begin{proof}[Proof of Proposition~\ref{prop:rand-error-G-iii-iv}]
\emph{(i)}
The bias--variance decomposition $\bar g(x)=\grad f(x)+\bias(x)$ and Proposition~\ref{prop:rand-error-G-i-ii},
together with $\sqrt{a+b}\le\sqrt a+\sqrt b$, give
\[
\begin{aligned}
  \E[&\norm{\gest(x)-\grad f(x)}^2]^{1/2}
  \le
  \norm{\bias(x)}
  +
  \sqrt{
    \E[\norm{\gest(x)-\bar g(x)}^2]
  } \\
  &\le
  \frac{\mu^2}{6}d(d+2)M_3
  +
  \sqrt{
    \frac{2(d+2)}{m}\norm{\grad f(x)}^2
    +
    \frac{\mu^4}{18m}
    d(d+2)(d+4)(d+6)M_3^2
  },
\end{aligned}
\]
which yields the stated estimate.

\emph{(ii)}
From \eqref{eq:gauss-inner-moments} we have
\(
  \E[A]=a,
  \;
  \E\norm A^2=(d+2)\norm a^2
\)
for the Gaussian distribution
and hence
\(
  \Var(A)=(d+1)\norm a^2
\)
is satisfied. Moreover, \eqref{eq:B-variance-bound} and the Gaussian eighth moment give
\[
  \sqrt{\Var(B)}
  \le
  \frac{\mu^2M_3}{6}
  \sqrt{d(d+2)(d+4)(d+6)}.
\]
Thus, by \eqref{eq:variance-perturbation} and the independence of the random variables, the inequality 
\[
  \E[\norm{\gest(x)-\bar g(x)}^2]^{1/2}
  \ge
  \frac1{\sqrt m}
  \left[
    \sqrt{d+1}\norm a
    -
    \frac{\mu^2M_3}{6}
    \sqrt{d(d+2)(d+4)(d+6)}
  \right]_+
\]
is satisfied. Finally, the bias--variance decomposition gives
\(
  \E[\norm{\gest(x)-\grad f(x)}^2]
  \ge
  \E[\norm{\gest(x)-\bar g(x)}^2],
\)
which transfers the bound to the RMS error. The final claim follows immediately
from the imposed condition on $\mu$.
\end{proof}

\begin{proof}[Proof of Proposition~\ref{prop:rand-error-S-iii-iv}]
\emph{(i)}
The bias--variance decomposition $\bar g(x)=\grad f(x)+\bias(x)$ and Proposition~\ref{prop:rand-error-S-i-ii}
give
\[
\begin{aligned}
  \E[\norm{\gest(x)-\grad f(x)}^2]^{1/2}
  &\le
  \norm{\bias(x)}
  +
  \sqrt{
    \E[\norm{\gest(x)-\bar g(x)}^2]
  } \\
  &\le
  \frac{d\mu^2}{6}M_3
  +
  \sqrt{
    \frac{2d}{m}\norm{\grad f(x)}^2
    +
    \frac{d^2\mu^4}{18m}M_3^2
  }.
\end{aligned}
\]

\emph{(ii)}
For the spherical distribution, we know from  \eqref{eq:moment2} that
\(
  \E[A]=\frac1d a,
  \;
  \E\norm A^2=\frac1d\norm a^2,
\)
and hence
\(
  \Var(A)=\frac{d-1}{d^2}\norm a^2
\)
holds. 
Since $\norm u=1$, \eqref{eq:B-variance-bound} gives
\(
  \sqrt{\Var(B)}
  \le
  \frac{\mu^2}{6}M_3.
\)
Accordingly, by \eqref{eq:variance-perturbation} and the independence of the random variables,
\[
  \E[\norm{\gest(x)-\bar g(x)}^2]^{1/2}
  \ge
  \frac1{\sqrt m}
  \left[
    \sqrt{d-1}\norm a
    -
    \frac{d\mu^2}{6}M_3
  \right]_+
\]
follows.
The bias--variance decomposition transfers this bound to the RMS error, and the
final claim follows from the stated restriction on $\mu$.
\end{proof}

\end{appendices}

\bibliography{main}

\end{document}